\documentclass[12pt,a4paper,reqno]{amsart}
\usepackage[utf8]{inputenc}
\usepackage{mathrsfs, amssymb, amsmath, amsfonts, enumitem}
\usepackage{graphicx}
\usepackage[english]{babel}
\title[Optimal Frac. $p$-Hardy Inequality]{Optimal Hardy Inequalities for Fractional $p$-Laplacians on the Integers}

\author{Florian Fischer}
\address{Florian Fischer, Institute for Applied Mathematics, University of Bonn, Endenicher Allee 60, 53115 Bonn, Germany}
\email{fischer@iam.uni-bonn.de, florian.fischer.math@gmail.com}

\author{Marius Nietschmann}
\address{Marius Nietschmann, Sorbonne Université, Université Paris Cité, CNRS, IMJ-PRG, 75005 Paris, France}
\email{nietschmann@imj-prg.fr}

\usepackage[style=alphabetic,maxnames=5,minnames=2,abbreviate=false,date=long,url=true,isbn=false, doi=true,backend=biber, backref,giveninits=true, backrefstyle=none]{biblatex}
\usepackage{csquotes}
\AtBeginBibliography{\tiny}
\DefineBibliographyStrings{english}{%
	backrefpage = {\hspace{-0.4em}}, 
	backrefpages = {\hspace{-0.4em}} 
}

\usepackage[colorlinks=true,citecolor=magenta,linkcolor=green]{hyperref}

\newtheorem{theorem}{Theorem}[section]
\newtheorem{lemma}[theorem]{Lemma}
\newtheorem{proposition}[theorem]{Proposition}
\newtheorem{corollary}[theorem]{Corollary}

\theoremstyle{definition}

\newtheorem{remark}[theorem]{Remark}

\numberwithin{equation}{section}

\newcommand{\abs}[1]{\left\lvert #1\right\rvert} 
\newcommand{\set}[1]{\left\{ #1\right\} }

\newcommand{\p}[1]{\left( #1 \right)^{\langle p-1 \rangle}}
\newcommand{\sse}{\subseteq}

\renewcommand{\epsilon}{\varepsilon}
\newcommand{\ep}{\varepsilon}
\renewcommand{\phi}{\varphi}
\newcommand{\NN}{\mathbb{N}}
\newcommand{\ZZ}{\mathbb{Z}}

\newcommand{\RR}{\mathbb{R}}
\newcommand{\CC}{\mathbb{C}}
\newcommand{\dd}{\mathrm{d}}
\DeclareMathOperator{\sgn}{sgn}

\DeclareMathOperator{\supp}{supp}
\DeclareMathOperator{\Div}{div}
\newcommand{\FF}{\mathcal{F}}

\newcommand{\E}{\mathcal{E}}
\newcommand{\EE}{\mathbf{E}}
\newcommand{\LL}{\mathbf{L}}

\newcommand{\Hmm}[1]{\leavevmode{\marginpar{\tiny%
			$\hbox to 0mm{\hspace*{-0.5mm}$\leftarrow$\hss}%
			\vcenter{\vrule depth 0.1mm height 0.1mm width \the\marginparwidth}%
			\hbox to 0mm{\hss$\rightarrow$\hspace*{-0.5mm}}$\\\relax\raggedright #1}}}
\begin{document}
	\begin{abstract}
		We obtain optimal Hardy inequalities for fractional powers of the $p$-Laplacian $\Delta_p^\sigma$ on $\ZZ$, $p\in (1,\infty)$, $\sigma \in (0, 1/p)$. Asymptotically, the optimal weights behave like $\abs{x}^{-p\sigma}$ as $\abs{x}\to \infty$. We show a similar result for the Riesz fractional $p$-Laplacian.
		
		In an appendix, we show that for general $p$-Laplace operators on locally summable graphs, $p$-null-criticality implies $p$-optimality at infinity.
		\\
		\\[4mm]
		\noindent  2020  \! {\em Mathematics  Subject  Classification.}
		Primary   26D15; Secondary 26A33, 35J92
		\\[4mm]
	\end{abstract}
	
	\maketitle
	
	\section{Introduction}
	It is well-known that Brownian motion is \textit{transient} on $\RR^d$ if and only if $d\geq 3$. That is, a Brownian particle does not return almost surely and  escapes to infinity with positive probability. It was observed some decades ago (see e.g. \cite{Gri} for historical details) that this transience on $\RR^d$ is equivalent to the existence of a Hardy inequality for the associated Dirichlet energy on $C_c^{\infty}(\RR^d)$, i.e., $\Delta \geq W$ in the sense of quadratic forms for some non-trivial non-negative function $W$, and $\Delta$ is the positive Dirichlet Laplacian. Here, we want to treat the corresponding quasi-linear discrete setting.
	
	A century ago, Landau \cite{Landau} showed the following improvement of an inequality by Hardy \cite{Hardy20}, \cite[Theorem 326]{HLP}:
	
	\[ \sum_{n=0}^{\infty}\abs{\phi(n+1)-\phi(n)}^p \geq \left(\frac{p-1}{p}\right)^p \sum_{n=1}^{\infty}\frac{\abs{\phi(n)}^p}{n^p}, \]
	where $\phi\in C_c(\NN)$, $\phi(0)=0$, $p>1$. It was observed in \cite{FKP} that although the $p$-Hardy constant $C_p:=\left(\frac{p-1}{p}\right)^p$ is not improvable, the whole $p$-Hardy weight $w_H(n):=C_p/n^p$ is not optimal, and the optimal improvement is given by $w_H$ plus lower-order terms. The basic idea in \cite{FKP} has its origins in its counterpart in the continuum in \cite{DP16} and was generalized to other graphs in \cite{F:Thesis, F} building upon the linear ($p=2$) theory \cite{DFP, KPPop, KPPN}. 
		
	However, for general $p$-Laplacians on graphs, only results about optimal $p$-Hardy weights are known in the setting of locally finite graphs with proper Agmon ground states of bounded oscillation. Here, we give for the first time an example of a non-locally finite graph admitting an optimal $p$-Hardy weight for $p\neq 2$, see Theorem~\ref{thm:main}. Remarkably, our corresponding Agmon ground state is proper but not of bounded oscillation.
	
	Our model is a special complete graph over $\ZZ$: the graph associated with the fractional (power of the) $p$-Laplacian. This operator has been studied for a long time in the continuum (see e.g. \cite{dTGCV} and references therein) but also has attracted growing interest in the discrete setting, see e.g. \cite{DKP, Wang,Yu,  ZLY}. In the linear case, that is $p=2$, Keller and the second author showed in \cite{KN} the optimality of a Hardy weight for $\Delta_2^{\sigma}$, $\sigma \in (0,1/2)$, obtained previously in \cite{CR}, see also \cite{DFF, GKS, SW26} for similar results on $\NN$.
	
	Our strategy is motivated by the results  for the linear case in \cite{KN}: We propose a candidate for the Agmon ground state, and show that it is indeed a suitable candidate to obtain a critical Hardy-type inequality by constructing a null-sequence. Since the obtained weight is in general not positive everywhere, we also propose a modification yielding a strictly positive Hardy weight. Heuristically, a reasonable candidate for an Agmon ground state for a $p$-Laplacian is often $G^{(p-1)/p}$, where $G$ is the Green's function. We view our candidate as closely related to the Green's function of our fractional $p$-Laplacian where we take the root not as a power but in a certain fractional sense.
	
	There is a closely related restricted (or Riesz) fractional $p$-Laplacian which coincides with the fractional power of the $p$-Laplacian on $\RR^d$, \cite{dTGCV}.  However, apart from the Euclidean setting these operators are different. Recently, Dyda showed the existence of a corresponding $p$-Hardy inequality on $\ZZ^d$, \cite{Dyda}. This corresponds to a graph where the edge weights are chosen to be $\abs{x-y}^{-d-p\sigma}$ for all $x,y\in \ZZ^d$. By adapting our previous approach, we briefly show in Theorem~\ref{thm:rfl} how to obtain optimal weights on $\ZZ$ associated with the restricted fractional $p$-Laplacian. 
	
	The linear fractional case on $\ZZ$ has been studied much more intensively; we mention here a few recent contributions \cite{BGKS,CR, DFF, FBRR,GKS, HKP, HKPLattice, HKPFractional, ZYY}. For applications of Hardy inequalities we refer to the references in \cite{BDK, BEL, OWR, D99, FS}. We only want to mention that the existence of a Hardy inequality is equivalent to the hyperbolicity of the graph \cite{AFS}, and in a probabilistic setting, to the transience of the corresponding random walk \cite{KLW, Soardi, Woess}.
	
	The paper is organized as follows: In the next section, we introduce the setting, state the main results, Theorem~\ref{thm:main} and Theorem~\ref{thm:strictly positive weight}, recall the ground state representation formula and some properties of $\Gamma$-functions. In Section~\ref{s:sign}, we study the sign of the Laplacian applied to our candidate for a ground state. Thereafter, in Section~\ref{s:optimal}, we give a proof of Theorem~\ref{thm:main}. Section~\ref{s:constant} is dedicated to the $p$-Hardy constant. In Section~\ref{s:positive weight}, we address the question of constructing a strictly positive weight via a proof of Theorem~\ref{thm:strictly positive weight}. In an appendix, we first show that $p$-null-criticality implies $p$-optimality at infinity, and then the associated theorem for the restricted fractional $p$-Laplacian.
	
	\section{Setting the Scene}\label{s:setting}
	The fractional $p$-Laplacian on $\ZZ$ is a special locally summable but non-locally finite graph. Since in the appendix we show a statement for $p$-Laplace operators on general graphs we introduce them here first and thereafter in Subsection~\ref{ss:frac}, we turn to the fractional setting on~$\ZZ$.
	
	\subsection{Discrete $p$-Laplacians on Graphs}\label{s:graphs}
	
	Let $X$ be countable. Elements of $X$ are called \emph{vertices}. Let $m\colon X\to (0,\infty)$. Then $(X,m)$ can be interpreted as a discrete measure space. The function $b\colon X\times X \to [0,\infty)$ is called a \emph{graph} over $(X,m)$ if $b$ is symmetric, vanishes on the diagonal and is locally summable, i.e.,
	\[ \deg(x):=\sum_{y\in X}b(x,y)< \infty, \qquad x\in X. \]
	Moreover, we assume that the graph is connected, that is, for any two vertices $x,y\in X$ there are vertices ${x_0,\ldots ,x_n \in X}$, such that $x=x_0$, $y=x_n$ and $b(x_{i-1}, x_i)> 0$ for all $i\in\set{1,\ldots, n}$. A graph is \emph{locally finite} if $\# \set{y\in X: b(x,y)>0}< \infty$ for all $x\in X$.
	
	By $C(X)$ and $C_c(X)$, we denote the spaces of real-valued and of finitely supported functions on $X$, respectively. By $1_K$ we denote the characteristic function which is $1$ on $K\sse X$ and $0$ elsewhere. For singletons~$\set{x}$, we do not write the braces, i.e., $1_x=1_{\set{x}}$.  We also define the linear difference operator $\nabla { \colon C(X)\to C(X\times X)}$ by the formula
	\[\nabla_{x,y}f=f(x)-f(y), \qquad x,y\in X, f\in C(X).\]
	
	Let $p\in[1,\infty)$ and let the \emph{formal space} $\FF_p=\FF_{p,b}$ be given by
	\begin{align*}
		\FF_p := \{ f\in C(X): \sum_{y\in X} b(x,y)\abs{\nabla_{x,y}f}^{p-1} < \infty  \mbox{ for all } x\in X  \}.
	\end{align*}
	If the graph is locally finite, we have $\FF_p=C(X)$. 
	
	The  \emph{(pseudo) $p$-Laplace operator} $\Delta_p\colon \FF_p\to C(X)$ is given by
	\begin{align*}
		\Delta_pf(x):=\frac{1}{m(x)} \sum_{y\in X} b(x,y)\p{\nabla_{x,y}f},\qquad x\in X.
	\end{align*}
	 Here, we used the French power notation, i.e.,
	 \[ \p{t}:= |t|^{p-1} \sgn (t), \qquad t\in \RR \]
	 with the sign function $\sgn\colon \RR\to \set{-1,0,1}$ such that $\sgn(t)=1$ for all $t > 0$, $\sgn(t)=-1$ for all $t< 0$, and $\sgn(0)=0$.
	 
	 The function $u\in\FF_p$ is $\Delta_p$-superharmonic on $X$ if $\Delta_p u \geq 0$ on $X$. Whenever the operator is clear from the context we will only speak of a superharmonic function. Furthermore, $u$ is subharmonic if $-u$ is superharmonic, and harmonic if it is super- and subharmonic at the same time.
	
	The \emph{$p$-energy functional} $\E_p\colon C_c(X)\to \RR$  
	is defined via 
	\[ \E_p(f):=\frac{1}{2}\sum_{x,y\in X} b(x,y)\abs{\nabla_{x,y}f}^{p}.\]
	The following Green's formula holds: for all $\phi\in C_c(X)$, \[\E_p(\phi)=\sum_{x\in X}\Delta_p\phi(x)\phi(x)m(x).\]
	
	\begin{remark}[$p$-Laplacians]\label{r:pLapl}
		\begin{enumerate}
			\item This discrete $p$-Laplacian has an analogue in the continuum -- the so-called \emph{pseudo $p$-Laplacian} $\tilde{\Delta}_p$, see e.g. \cite{BK}. On domains in $\RR^d$ it acts on smooth enough functions $u$ via \[\tilde{\Delta}_pu = - \sum_{i=1}^d\frac{\partial}{\partial x_i}\left(\abs{\frac{\partial}{\partial x_i}}^{p-2}\frac{\partial u}{\partial x_i}\right).\]
			\item\label{r:pLaplGri} Another generalization of the linear (that is $p=2$) Laplacian is motivated by a Green's formula	from the $p$-energy
			\[ \sum_{x\in X} \frac 1{2}\abs{\sum_{y\in X}b(x,y)(f(x)-f(y))^2}^{p/2},\]
			where the inner summation represents the square of the norm of the ``gradient''  of $f $ at $x$, see e.g. \cite{GLY}. While this description resembles the classical $p$-Laplacian $-\Div(\abs{\nabla}^{p-2}\nabla)$ in the continuum, it will not be considered here as the ``pseudo'' variant of the linear Laplacian has been studied much more on graphs and goes back at least to \cite{Y}.
		\end{enumerate}
	\end{remark}
		
	For any $w\in C(X)$, we use the following notation (whenever the sum converges absolutely):
	\[ w_p(f):= \sum_{x\in X}w(x)\abs{f(x)}^{p}m(x).\]
	In other words, $w_p$ is the canonical functional associated with the function $w$.
	
	We say that $w$ is a \emph{($p$-) Hardy-type weight} if 
	\[ \E_p(\phi)\geq w_p(\phi), \qquad \phi \in C_c(X).\]
	Since $\E_p \geq 0$, the trivial function $w=0$ is always a Hardy-type weight and it is therefore desirable to find those weights which are positive somewhere. A positive Hardy-type weight is called a \emph{Hardy weight}. 
	
	A Hardy-type weight $w$ is called \emph{$p$-critical} if there is no other Hardy-type weight $\tilde{w}$ such that $\tilde{w} \gneq w$ on $X$. This is equivalent to the existence of an Agmon ground state, that is, the unique (up to multiplicative constants) positive ($\Delta_p -w$)-superharmonic function $u$. In fact $u$ is ($\Delta_p -w$)-harmonic, i.e., $\Delta_pu=w u^{p-1}$, see \cite{F:GSR} for a proof in the language of non-negative $p$-Schrödinger operators. 
	
	A $p$-critical Hardy-type weight $w$ is \emph{$p$-null-critical} if the corresponding Agmon ground state $u$ is not in $\ell^{p}(X,\abs{w}m)$, i.e., $w_p(u)=\infty$ formally. Thus, equality can only occur if both sides are $\infty$ or $0$, and therefore it is also sometimes called non-attainability \cite{GKS}.
	
	A Hardy-type weight is called \emph{optimal} if it is $p$-critical and $p$-null-critical. Historically, a third requirement for optimality is usually stated: $p$-optimality at infinity. We show in the appendix that $p$-null-criticality implies $p$-optimality at infinity.
	
	\subsection{Fractional $p$-Laplacians on $\ZZ$}\label{ss:frac}
	The linear Laplacian over $\ZZ$ is given by
		\[ \Delta_2f(x)= (f(x)-f(x+1))+ (f(x)-f(x-1)),\]
	for $x\in \ZZ, f\in C(\ZZ)$. This corresponds to setting $p=2, m=1$ and $b(x,y)=1$ if $\abs{x-y}=1$ and $b(x,y)=0$ else, $x,y\in \ZZ$, in Section~\ref{s:graphs}. 
	
	Let $p\in (1,\infty)$ and $\sigma\in (0,\frac2p)$. The \emph{fractional $p$-Laplacian} $\Delta_p^{\sigma}$ is then defined by
		\[ \Delta_p^\sigma f(x):= \frac{1}{\abs{\Gamma\left(\frac{-p\sigma}{2}\right)}}\int_0^\infty e^{-t\Delta_2}(\p{\nabla_{x \cdot}f})(x) \frac{\dd t}{t^{1+\frac{p\sigma}{2}}} \]
	for $f\in C_c(\ZZ)$.	Here, $e^{-t\Delta_2}, t\geq 0$, is the semigroup of $\Delta_2$ and $\Gamma$ is the Gamma function. For $p=2$, it is the classical fractional Laplacian on~$\ZZ$, see e.g. \cite{KN}. This operator is also called the fractional power of the $p$-Laplacian or the spectral fractional $p$-Laplacian.
	
	We show that it can be interpreted as a discrete graph Laplacian. Let $p_t(x,y)= e^{-\Delta_2}1_y(x)$, $x,y\in\ZZ$, be the heat kernel of $\Delta_2$, and set for all $\beta \in (-\frac1p,\frac2p)$, $\beta\neq0$, 
	
	\[\kappa_{p,\beta}(x,y):=\frac{1}{\abs{\Gamma\left(\frac{-p\beta}{2}\right)}}\int_0^\infty p_t(x,y)\frac{\dd t}{t^{1+\frac{p\beta}{2}}},\]
	for $\abs{x-y}>\frac{p\beta}2$ as well as $\kappa_{p,\beta}(x,x)=0$ for $\beta > 0$, and $\kappa_{p,0}=1_0$. From symmetry properties of the heat kernel, we infer	
	\[ \kappa_{p,\beta}(x,y)=\kappa_{p,\beta}(x-y,0)=:\kappa_{p,\beta}(x-y).\]
	
	By \cite[Lemma~9.2]{CRSTV}, we have the following asymptotic behavior:
	\begin{align}\label{eq:kappa}
		\kappa_{p,\beta}(x)&= \frac{4^{\frac{p\beta}{2}} \Gamma(\frac{1+p\beta}{2})}{\sqrt{\pi}\abs{\Gamma(- \frac{p\beta}{2})}}\cdot \frac{\Gamma(\abs{x}-\frac{p\beta}{2})}{\Gamma (\abs{x}+1+\frac{p\beta}{2})} \\
		&= c_{p,\beta} \cdot \abs{x}^{-1-p\beta}+ O\left(\abs{x}^{-2-p\beta} \right), \quad |x|\to \infty.\label{eq:kappaAsymp}
	\end{align}
	with 
	\begin{equation} \label{eq:kappa constant}
		c_{p,\beta} := \frac{4^{\frac{p\beta}{2}} \Gamma(\frac{1+p\beta}{2})}{\sqrt{\pi}\abs{\Gamma(- \frac{p\beta}{2})}}.
	\end{equation}
	
	Let us now consider the graph $\kappa_{p,\sigma}$ over $(\ZZ, 1)$. This is a non-locally finite complete graph which is locally summable since $p\sigma >0$. Thus, for all $f\in C_c(\ZZ)$, we have
	\begin{align*}
		\sum_{y\in \ZZ}\kappa_{p,\sigma}(x,y) &\p{\nabla_{x,y}f}\\
		&= \sum_{y\in \ZZ}\left(\frac{1}{\abs{\Gamma\left(\frac{-p\sigma}{2}\right)}}\int_0^\infty p_t(x,y)\frac{\dd t}{t^{1+\frac{p\sigma}{2}}}\right) \p{\nabla_{x,y}f}\\
		&= \frac{1}{\abs{\Gamma\left(\frac{-p\sigma}{2}\right)}} \int_0^\infty \sum_{y\in \ZZ} p_t(x,y) \p{\nabla_{x,y}f}\frac{\dd t}{t^{1+\frac{p\sigma}{2}}} \\
		&= \frac{1}{\abs{\Gamma\left(\frac{-p\sigma}{2}\right)}}\int_0^\infty e^{-t\Delta_2}(\p{\nabla_{x \cdot}f})(x) \frac{\dd t}{t^{1+\frac{p\sigma}{2}}} \\
		&= \Delta_p^\sigma f(x),
	\end{align*}
	see \cite[Section~2]{KN} for analogous computations in the case $p=2$. We extend $\Delta_p^\sigma$ to functions in $\FF_p=\FF_{p,\kappa_{p,\sigma}}$ via the graph $p$-Laplacian description.
	
	\begin{remark}[fractional $p$-Laplacians]
		\begin{enumerate}\label{r:fracpLapl}
			\item In \cite{CG}, an alternative generalization of the action of a fractional Laplacian is given via
			\[ \int_0^\infty \left(1- e^{-t\Delta_p}\right)f(x)\frac{\dd t}{t^{1+\sigma}}, \qquad x\in \ZZ, \]
			where $e^{-t\Delta_p}$ is the non-linear semigroup associated with the $p$-Laplacian, see also \cite{M}. This is also sometimes called a fractional $p$-Laplacian but will not be considered here.
			\item Following the approach in the previous Remark \ref{r:pLapl}~\eqref{r:pLaplGri} via a different $p$-Laplacian, an alternative fractional $p$-Laplacian on locally finite graphs is defined in \cite{ZLY}.
			\item Moreover, there is another very popular fractional $p$-Laplacian, sometimes called restricted (or integral or Riesz) fractional $p$-Laplacian, which we will discuss in detail in the appendix. This operator can be obtained by choosing $b(x,y)=\abs{x-y}^{-1-p\sigma}$. It is well known that our spectral fractional $p$-Laplacian and this restricted one coincide if the underlying space is Euclidean \cite{dTGCV}.
		\end{enumerate}		
	\end{remark}
	
	\subsection{Main Result}
	We have the following result, which is the direct generalization of \cite[Theorem~5]{KN}.
	
	\begin{theorem}\label{thm:main}
		Let $p,q \in (1,\infty)$, $\sigma\in (0,\frac1p)$, $\alpha\in(0,\frac1q)$. Let $\kappa_{p,\sigma}$ be a graph over $(\ZZ,1)$. Set 
		\[ w:= \frac{\Delta_p^\sigma \kappa_{q,-\alpha}}{\kappa_{q,-\alpha}^{p-1}}. \]
			\begin{enumerate}[label = ({\alph*})]
			
			\item\label{thm:main_a} If $p(\sigma+1-q\alpha)\geq1$, i.e., $\alpha\leq\frac{p-1+p\sigma}{pq}$, then $w$ is $p$-critical.
			\item\label{thm:main_b} If  $p(\sigma+1-q\alpha)\leq1$, i.e., $\alpha\geq\frac{p-1+p\sigma}{pq}$, then $\kappa_{q,-\alpha}\notin \ell^p(\ZZ,\abs{w})$.
		\end{enumerate}
	In particular, if $p(\sigma+1-q\alpha)=1$, then $w$ is an optimal Hardy-type weight.
	
	Moreover, $w(0)>0$. If $\alpha>\frac{p-2+p\sigma}{(p-1)q}$, then $w>0$ outside a compact set and $w(x)\asymp |x|^{-p\sigma}$ as $|x|\to\infty$. 
	\end{theorem}
	
	A key ingredient is the correct asymptotic behaviour of $w$ and $\kappa_{q,-\alpha}$. This will be used in the appendix to show an alternative optimality result for the restricted fractional $p$-Laplacian.
	
	Another key ingredient is the application of the ground state representation formula, which requires $w$ to be given by a quotient as in the theorem. In the case $p=q=2$, one has the powerful identity
	\[ \Delta_2^\sigma \kappa_{2,-\alpha}=\kappa_{2,\sigma-\alpha}, \]
	which is proven via the spectral theorem and crucial for establishing the positivity of the associated Hardy weight for $\sigma\leq\alpha$. No analogous identity is available in the non-linear case; in Section~\ref{s:sign} we show some partial results on the sign of $\Delta_p^\sigma\kappa_{q,-\alpha}$. We will see in particular that the Hardy-type weight is in general not a Hardy weight. 
	
	Nevertheless, it would be very interesting to determine whether a similar formula could also hold in the non-linear case.
	
	Next, we formulate our second main result addressing the question of a strictly positive Hardy weight.
	
	\begin{theorem}\label{thm:strictly positive weight}
		Let $p,q>1$, $\sigma\in(0,\frac1p)$ and $\alpha=\frac{p-1+p\sigma}{pq}$. There exists $\ep>0$ such that $u=\kappa_{q,-\alpha}\wedge\ep$ is $\Delta_p^\sigma$-superharmonic. In particular, 
			\[ w := \frac{\Delta_p^\sigma(\kappa_{q,-\alpha}\wedge\ep)}{(\kappa_{q,-\alpha}\wedge\ep)^{p-1}} \] 
		is an optimal Hardy weight. Moreover, $w(x)\asymp|x|^{-p\sigma}$ as $|x|\to\infty$. 
	\end{theorem}
		
	\subsection{Toolbox}
	\subsubsection{Ground state representation}\label{s:GSR}
	The main tool for obtaining our results is the so-called ground state representation formula from \cite[Theorem~3.1]{F:GSR}. We restate it here for convenience. We also formulate it in the more general setting of locally summable graphs since we use it in the appendix when showing that null-criticality implies optimality at infinity. 
	
	Let $p>1$ and $0\leq u\in \FF_p$. The \emph{simplified energy (functional)} $\E_{p,u}$ of $\E_p$ with respect to $u$ on $C_c(X)$ is defined as
	\begin{align*}
		\E_{p,u}(\phi)&:=\sum_{x,y\in X}b(x,y) u(x)u(y)(\nabla_{x,y}\phi)^{2} \\
		&\qquad \cdot\left( \bigl(u(x)u(y)\bigr)^{1/2}\abs{\nabla_{x,y}\phi}+ \frac{\abs{\phi(x)}+ \abs{\phi(y)}}{2}\abs{\nabla_{x,y}u} \right)^{p-2},
	\end{align*}
	where we set $0\cdot \infty =0$ if $1<p<2$. By $a\asymp b$ we  mean that there are positive constants $C_1, C_2$ such that $C_1 b\leq a \leq C_2 b$ for non-negative quantities $a,b$.
	
	\begin{theorem}[Ground state representation, {\cite[Theorem~3.1]{F:GSR}}]\label{thm:GSR}
		Let $p> 1$, and $0\leq u\in \FF_p$. Then, we have
		\begin{align}\label{eq:GSRI}
			\E_p(u\phi)- (u\Delta_pu)_p(\phi)\asymp \E_{p,u}(\phi), \qquad \phi\in C_c(X).
		\end{align}
		Furthermore, for $p=2$, the equivalence becomes an equality.
	\end{theorem}
	
	Of particular interest is the following special case of the ground state representation: Assume that $0< u\in \FF_p$   and set $w:= \Delta_p u/u^{p-1}$, then Theorem~\ref{thm:GSR} yields the following Hardy-type inequality,
	\[ \E_p(\phi)\geq  w_p(\phi), \qquad \phi\in C_c(X).\]
	
	From the inequalities in Theorem~\ref{thm:GSR}, we consequently obtain estimates comparing the energy associated with the $p$-Laplace operator with other functionals, which are usually also referred to as \emph{simplified energies} (see e.g. \cite{DP16, PTT08}). They are all called simplified, because they consist of non-negative terms only, and the difference operator $\nabla$ applies either to $u$ or $\phi$ but not to the product $u\cdot \phi$.
	
	We set on $C_c(X)$
	\begin{align*}
		\E_{p,u,1}(\phi):= \sum_{x,y\in X}b(x,y) (u(x)u(y))^{p/2}\abs{\nabla_{x,y}\phi}^{p},
	\end{align*}
	and for $p\geq 2$, we define on $C_c(X)$
	\[\E_{p,u,2}(\phi):=\sum_{x,y\in X}b(x,y)u(x)u(y)\abs{\nabla_{x,y}u}^{p-2} \left(\frac{\abs{\phi(x)}+ \abs{\phi(y)}}{2}\right)^{p-2}\abs{\nabla_{x,y}\phi}^{2}.\]
	
	\begin{corollary}[{\cite[Corollary~3.2]{F:GSR}}]\label{cor:GSR}
		If $1<p\leq 2$, then there is a positive constant $C_{p}$ such that for all $0\leq u\in \FF_p$
		\begin{align}\label{eq:GSRI_p<2}
			\E_p(u\phi) - (u\Delta_pu)_p(\phi)\leq C_p \E_{p,u,1}(\phi),\quad \phi\in C_c(X).
		\end{align}
		and if $p\geq 2$, the reversed inequality in \eqref{eq:GSRI_p<2} holds true, i.e., 
		\begin{align}\label{eq:GSRI_p>2}
			\E_p(u\phi) - (u\Delta_pu)_p(\phi)\geq C_p \E_{p,u,1}(\phi),\quad \phi\in C_c(X).
		\end{align}
		Furthermore, both inequalities become equalities if $p=2$.
		
		Moreover, if $p\geq 2$, we have for all $0\leq u\in \FF_p$ and $\phi\in C_c(X)$
		\begin{align}\label{eq:GSRI_p>2Triangle}
			\E_p(u\phi)- (u\Delta_pu)_p(\phi) \asymp \E_{p,u,1}(\phi)+\E_{p,u,2}(\phi), \qquad \phi\in C_c(X).
		\end{align}
	\end{corollary}

	Furthermore, we will make use of the following characterization of $p$-criticality. It is the direct non-linear generalization of \cite[Theorem~5.3]{KPP}.
	\begin{lemma}[Null-sequence, {\cite[Theorem~3.3 and Appendix~A]{F}}]\label{lem:null-sequence}
	Let $w$ be a Hardy-type weight. Then the following are equivalent:
	\begin{enumerate}[label = ({\roman*})]	
		\item The Hardy-type weight $w$ is $p$-critical.
		\item There exists a unique (up to scalar multiplication) function $0\lneq v \in \FF_p$ such that $\Delta_p v \geq  wv^{p-1}$. This function is $(\Delta_p-w)$-harmonic and called the \emph{Agmon ground state}.
		\item There exists a sequence $(e_n )$ in $C_c (X)$ such that $(\E_p - w_p)(e_n ) \to 0$ and such that
		$(e_n (o))$ converges to a non-zero constant for some $o \in X$. Such a sequence is called a \emph{null-sequence}.
		\item There exists a null-sequence $(e_n )$ in $C_c (X)$ such that
		$(e_n )$ converges pointwise from below to a function $0\lneq v \in \FF_p$ with $\Delta_p v=wv^{p-1}$.
	\end{enumerate}
	\end{lemma}
	Alternatively, the statement can also be deduced from \cite[Theorem~4.1]{F:GSR} together with \cite[Lemma~2.4]{DKP}. For the case $p=2$ see also \cite{HKPLattice,HKP,  Hake, KLW}.
	
	As a consequence of the ground state representation formula, we get the following result.
	\begin{corollary}\label{cor:null}
		Let $w=\frac{\Delta_p u}{u^{p-1}}$ be a Hardy-type weight for some function $0< u \in \FF_p$. Then, $w$ is $p$-critical if and only if there exists a sequence $(e_n)$ in $C_c(X)$ such that $e_n\to 1$ pointwise and $\sup_n \E_{p,u}(e_n)<\infty$.
	\end{corollary}
	\begin{proof}
		By the null-sequence characterization in Lemma~\ref{lem:null-sequence}, ``$\implies$'' is a trivial consequence of Theorem~\ref{thm:GSR}. The converse follows from \cite[Proposition~3.3]{AFS}, that is, the assumption is equivalent to the existence of a null-sequence for $\E_{p,u}$, and by Theorem~\ref{thm:GSR} the result follows.
	\end{proof}

	\subsubsection{Properties of the $\Gamma$-function}
	The edge weights of the graph corresponding to the fractional $p$-Laplacian are given via ratios of $\Gamma$-functions, see Eq.~\eqref{eq:kappa}. It is well known that the $\Gamma$-function is a meromorphic function satisfying the functional equation
		\[ \Gamma(1+z) = z\Gamma(z) \] 
	for all $z\in\CC$. Furthermore, $\Gamma$ is continuous on $(0,\infty)$ and $\Gamma(1)=1$. 
	
	\begin{lemma}\label{lem:Gamma}
		Let $x\in\NN$, $y\in\NN_0$, $s\in(0,1)$ and $t\in(0,1/2)$. Then, the quotients 
			\[ \frac{\Gamma(x-s)}{\Gamma(x+1+s)} 
				\qquad\text{and}\qquad
				\frac{\Gamma(y+t)}{\Gamma(y+1-t)} \] 
		are strictly decreasing in $x$ and $y$ respectively. 
	\end{lemma}
	\begin{proof}
		Using the functional equation of the Gamma function, we have 
		\[ \frac{\Gamma(y+1+t)}{\Gamma(y+2-t)} 
			= \frac{y+t}{y+1-t} \frac{\Gamma(y+t)}{\Gamma(y+1-t)} 
			< \frac{\Gamma(y+t)}{\Gamma(y+1-t)} \] 
		since $t\in(0,1/2)$. The other computation is similar. 
	\end{proof}

	\section{On the sign of  $\Delta_p^\sigma\kappa_{q,-\alpha}$}\label{s:sign}
	First, we show that $\Delta_p^\sigma\kappa_{q,-\alpha} \geq 0$ outside of a compact set using the asymptotics in Eq.~\eqref{eq:kappaAsymp}.  We then use a more technical approach based on Lemma~\ref{lem:Gamma} to prove that it is superharmonic at $0$ but sometimes subharmonic at $\pm 1$. Hence, in general, we will only have Hardy-\emph{type} weights on $\ZZ$. This is different from the case $p=q=2$, where it is known that $\Delta_2^\sigma\kappa_{2,-\alpha} \geq 0$ on $\ZZ$ as long as $\sigma > \alpha$, see \cite[Proposition~2]{KN} (which uses the spectral theorem) and \cite{CR} (which uses the Fourier transforms).
	
	\begin{proposition}[Superharmonicity outside a compact set]\label{prop:superharmonic}
		Let $p,q>1$, $\sigma\in(0,\frac1p)$ and $\alpha\in(0,\frac1q)$. If $\alpha>\frac{p-2+p\sigma}{(p-1)q}$, then there exists a finite set $K\subset\ZZ$ such that $\Delta_p^\sigma \kappa_{q,-\alpha}(x) > 0$ for all $x\notin K$. In particular, $\frac{\Delta_p^\sigma \kappa_{q,-\alpha}}{\kappa^{p-1}_{q,-\alpha}}$ is a non-trivial Hardy-type weight. 
	\end{proposition}
	
	\begin{proof}
		We set $u = \kappa_{q,-\alpha}$. By symmetry of $u$, we may assume $x$ to be positive. We consider  
			\[ \Delta_p^\sigma u(x) 
					+ \kappa_{p,\sigma}(x)(\nabla_{0,x} u)^{p-1}
				= \sum_{y\neq0} \kappa_{p,\sigma}(x-y) 
					(\nabla_{x,y} u)^{\langle p-1\rangle}. \] 
		We show that the latter is a positive function under the assumption on $\alpha$. We will then conclude superharmonicity outside a compact set by showing that $\kappa_{p,\sigma}(x)(\nabla_{0,x} u)^{p-1}$ decays faster than the right-hand side. Noting that, on the right-hand side, the terms where $|y|=x$ vanish, we split the sum into sets over $|y|<x$ and $|y|>x$, then regroup into sums over $0<y<x$ and $y>x$ to obtain 
		\begin{align*}
			&\sum_{y>x} (\kappa_{p,\sigma}(x-y) + \kappa_{p,\sigma}(x+y)) 
						(\nabla_{x,y} u)^{p-1} \\ 
				&\qquad- \sum_{0<y<x} (\kappa_{p,\sigma}(x-y) + \kappa_{p,\sigma}(x+y)) 
						(\nabla_{y,x} u)^{p-1}. 
		\intertext{Using the asymptotics of $\kappa_{p,\sigma}$  and $u$, cf. \eqref{eq:kappaAsymp}, we continue}
			&\gtrsim h(x)x^{-1}\left[
			\sum_{y>x} \left(\left(\tfrac yx-1\right)^{-1-p\sigma} 
						+ \left(\tfrac yx+1\right)^{-1-p\sigma}\right)
					\left(1-\left(\tfrac yx\right)^{-1+q\alpha}\right)^{p-1}\right. \\ 
				&\qquad- \left.\sum_{0<y<x} \left(\left(1-\tfrac yx\right)^{-1-p\sigma} 
						+ \left(1+\tfrac yx\right)^{-1-p\sigma}\right)
					\left(\left(\tfrac yx\right)^{-1+q\alpha}-1\right)^{p-1}\right] 
		\intertext{where we factored out 
			\[ h(x) = c_{p,\sigma} c_{q,-\alpha}^{p-1} x^{-p\sigma+(p-1)(-1+q\alpha)} \] 
		with $c_{p,\sigma}$ and $c_{q,-\alpha}$ from Equation~(\ref{eq:kappa constant}). Estimating the sums by integrals, we have}
			\cdots&\gtrsim h(x)x^{-1}\left[
			\int_x^\infty \left(\left(\tfrac yx-1\right)^{-1-p\sigma} 
						+ \left(\tfrac yx+1\right)^{-1-p\sigma}\right)
					\left(1-\left(\tfrac yx\right)^{-1+q\alpha}\right)^{p-1}dy\right. \\ 
				&\qquad- \left.\int_0^x \left(\left(1-\tfrac yx\right)^{-1-p\sigma} 
						+ \left(1+\tfrac yx\right)^{-1-p\sigma}\right)
					\left(\left(\tfrac yx\right)^{-1+q\alpha}-1\right)^{p-1}dy\right] 
		\intertext{now substituting $t=\frac xy$ and $s=\frac yx$, we obtain} 
			\cdots&\gtrsim h(x)\left[
			\int_0^1 \left(\left(t^{-1}-1\right)^{-1-p\sigma} 
						+ \left(t^{-1}+1\right)^{-1-p\sigma}\right)
					\left(1-t^{1-q\alpha}\right)^{p-1}t^{-2}dt\right. \\ 
				&\qquad- \left.\int_0^1 \left((1-s)^{-1-p\sigma} 
						+ (1+s)^{-1-p\sigma}\right)(s^{-1+q\alpha}-1)^{p-1}ds\right] \\ 
				&= h(x) \int_0^1 \left((1-t)^{-1-p\sigma} + (1+t)^{-1-p\sigma}\right) \\ 
					&\qquad\qquad\qquad\cdot
						\left(t^{1+p\sigma+(p-1)(1-q\alpha)-2}-1\right) 
						\left(t^{-1+q\alpha}-1\right)^{p-1} dt. 
		\end{align*}
		Denoting this integral by $I$, we have 
			\[ I>0 
				\iff 1+p\sigma+(p-1)(1-q\alpha)-2<0 
				\iff \alpha>\frac{p-2+p\sigma}{(p-1)q}. \] 
		Now, using again the asymptotics for large $x$, we have 
		\begin{align*}
			\Delta_p^\sigma u(x) 
				&\gtrsim h(x)I - \kappa_{p,\sigma}(x)(\nabla_{0,x} u)^{p-1} 
				\gtrsim h(x)\left(I-x^{-1}\left(\frac{u(0)}{u(x)}-1\right)^{p-1}\right) \\
				&\gtrsim h(x)(I-x^{-1+(p-1)(1-q\alpha)}) 
				\geq h(x)(I-x^{-p\sigma}) 
		\end{align*}
		where we used once more the assumption on $\alpha$ in the last step. Hence, for sufficiently large $x$, we have that $\Delta_p^\sigma u(x) > 0$. 
		\end{proof}
	
	Now we show that our candidate is actually positive at the origin. 
	\begin{lemma}[Superharmonicity at $0$]\label{lem:superharmonicat0}
		Let $p,q>1$, $\sigma\in(0,\frac1p)$ and $\alpha\in(0,\frac1q)$. Then, $\Delta_p^\sigma\kappa_{q,-\alpha}(0) >0$.
	\end{lemma}
	\begin{proof}
		From Lemma~\ref{lem:Gamma}, we infer that $\kappa_{q,-\alpha}(y)$ is decreasing in $y$ for $y\geq0$. Hence, due to symmetry, $\kappa_{q,-\alpha}$ has a maximum at $0$. 
		It follows that $\nabla_{0,y}\kappa_{q,-\alpha}>0$ for all $y\neq0$, and thus $\Delta_p^\sigma\kappa_{q,-\alpha}(0) >0$.
	\end{proof}

In the next lemma, we compute the vertex degree of the graph induced by the weights $\kappa_{p,\sigma}$. 

\begin{lemma}[Vertex degree]\label{lem:vertex degree}
	Let $p>1$ and $\sigma\in(0,\frac1p)$. Then, the vertex degree of the graph $(\ZZ,\kappa_{p,\sigma})$ is 
		\[ \deg(x) = c_{p,\sigma} \frac{\Gamma(1-\frac{p\sigma}2)}{\frac{p\sigma}2\Gamma(1+\frac{p\sigma}2)} = \left(1+\frac2{p\sigma}\right) \kappa_{p,\sigma}(1) \] 
	for any $x\in\ZZ$. 
\end{lemma}

\begin{proof}
	Since $\kappa_{p,\sigma}$ is shift invariant, the degree is constant. We set ${s:=\frac{p\sigma}2}$. We then have 
		\[ \deg(0) 
			= c_{p,\sigma} \sum_{x\neq0} \frac{\Gamma(|x|-s)}{\Gamma(|x|+1+s)} 
			= \frac{c_{p,\sigma}}{\Gamma(1+2s)} \sum_{x\neq0} B(|x|-s,1+2s) \] 
	where $B$ denotes the Beta-function. Computing the sum, we obtain 
	\begin{align*}
		\sum_{x\neq0} B(|x|-s,1+2s) 
			&= 2 \int_0^1 \sum_{x\geq1} t^{x-s-1} (1-t)^{1+2s-1} dt \\ 
			&= 2 \int_0^1 t^{1-s-1} (1-t)^{2s-1} dt \\ 
			&= 2 B(1-s,2s) 
			= 2 \frac{\Gamma(1-s)\Gamma(2s)}{\Gamma(1+s)}. 
	\end{align*}
	We therefore have 
		\[ \deg(0) = c_{p,\sigma} \frac{\Gamma(1-s)}{s\Gamma(1+s)}
			= \left(1+\frac1s\right) \kappa_{p,\sigma}(1). \] 
	This finishes the proof. 
\end{proof}

Next we show that if $\alpha$ is too small, $\kappa_{q,-\alpha}$ is subharmonic at $\pm1$. By symmetry it suffices to consider $x=1$ in the following.

\begin{proposition}
	[Subharmonicity at $\pm 1$ for small $\alpha$]\label{prop:subharmonicat1}
	Fix $p,q>1$ and $\sigma\in(0,\frac1p)$. Then, $\Delta_p^\sigma \kappa_{q,-\alpha}(\pm1)\leq0$ if 
		\[ \alpha\leq
			\frac2{q\left(2+\left(1+\frac2{p\sigma}\right)^{\frac1{p-1}}\right)}. \] 
\end{proposition}
\begin{proof}
	Without loss of generality, it suffices to consider $x=1$. The inequality $\Delta_p^\sigma \kappa_{q,-\alpha}(1)\leq0$ is equivalent to 
	\begin{align*}
		\sum_{y>1}(\kappa_{p,\sigma}(1-y)&+\kappa_{p,\sigma}(1+y)) \left(\kappa_{q,-\alpha}(1)-\kappa_{q,-\alpha}(y)\right)^{p-1} \\ 
			&\quad\leq \kappa_{p,\sigma}(1) \left(\kappa_{q,-\alpha}(0)-\kappa_{q,-\alpha}(1)\right)^{p-1}.
	\end{align*}
	We set $t:=\frac{q\alpha}2\in(0,\frac12)$. The left-hand side is bounded by 
		\[ \deg(1) \kappa_{q,-\alpha}(1)^{p-1} 
			= \deg(1) \left(\frac{4^{-t}\Gamma(\frac12-t)}{\sqrt\pi \Gamma(1-t)} \frac t{1-t}\right)^{p-1}, \] 
	where we used the functional equation of the $\Gamma$-function, whereas the right-hand side is equal to 
		\[ \kappa_{p,\sigma}(1)\left(\frac{4^{-t}\Gamma(\frac12-t)}{\sqrt\pi \Gamma(1-t)} \frac{1-2t}{1-t}\right)^{p-1}. \] 
	Rearranging shows that $\Delta_p^\sigma \kappa_{q,-\alpha}(1)\leq0$ if 
		\[ \left(\frac{\deg(1)}{\kappa_{p,\sigma}(1)}\right)^{\frac1{p-1}} 
			= \left(1+\frac2{p\sigma}\right)^{\frac1{p-1}} \leq \frac1t-2 \] 
	where the first equality is due to Lemma~\ref{lem:vertex degree}. Substituting $t=\frac{q\alpha}2$ and solving for $\alpha$ gives the claim. 
\end{proof}

\begin{remark}
	Note that optimal Hardy-type weights come from the choice $p(\sigma+1-q\alpha)=1$, i.e., $t=(\sigma+1-1/p)/2$. But then $\frac1t-2<\frac2{p-1}$, whereas $1+\frac2{p\sigma}>3$. It is not hard to see that $3^{\frac1{p-1}} > \frac2{p-1}$. 
	
	Moreover, in the case $p=q=2$, it is shown in \cite[Proposition~2]{KN} that $\kappa_{2,-\alpha}$ is $\Delta_2^\sigma$-superharmonic if $0<\sigma \leq \alpha$, which also indicates that $\alpha$ should not become too small relative to $\sigma$. 
	
	In Lemma~\ref{lem:sub_optimal}, we show that even for optimal parameter choices, the Hardy-type weight can become negative at $\pm1$. An alternative proof is given in the appendix.
\end{remark}

\begin{lemma}
	[Superharmonicity at $\pm 1$ for large $\alpha$ or small $p$]\label{lem:superharmonicat1_p}
	Fix $q>1$ and $\alpha\in  (0, \frac1q)$. Set $s=\frac{p\sigma}{2}$.
	If 	\[\alpha>\frac{4}
	{q\left(
		1+
		\left[
		1+\frac{(2-s)(1-s)}{(3+s)(2+s)}
		\right]^{\frac{1}{p-1}}
		\right)},\]
		then $\Delta_p^\sigma\kappa_{q,-\alpha}(\pm1) >0$. In particular, there exist $1<p_0<\infty$ and $0\leq\sigma_0<\frac{1}{p_0}$ such that for all $p \in (1,p_0)$ and $\sigma\in(\sigma_0,\frac1p)$ we have  $\Delta_p^\sigma\kappa_{q,-\alpha}(\pm1) >0$.
\end{lemma}
\begin{proof}
		Set 
	\[ s:=\frac{p\sigma}2, \qquad t:=\frac{q\alpha}2.\]
	Without loss of generality, let $x=1$. Moreover, $\Delta_p^\sigma\kappa_{q,-\alpha}(1) >0$ is equivalent to 
	\[  S_{\infty}+S_2> S_0,\]
	where
	\begin{align*}
		S_0&:=\frac{\Gamma(1-s)}{\Gamma(2+s)}\left(\frac{\Gamma(t)}{\Gamma(1-t)}- \frac{\Gamma(1+t)}{\Gamma(2-t)}\right)^{p-1}, \\
		S_2&:= \left(\frac{\Gamma(1-s)}{\Gamma(2+s)}+ \frac{\Gamma(3-s)}{\Gamma(4+s)}\right) 
		\left(\frac{\Gamma(1+t)}{\Gamma(2-t)}- \frac{\Gamma(2+t)}{\Gamma(3-t)}\right)^{p-1}\\
		S_{\infty}&:=\sum_{y>2} \left(\frac{\Gamma(y-1-s)}{\Gamma(y+s)}+ \frac{\Gamma(y+1-s)}{\Gamma(y+2+s)}\right) 
		\left(\frac{\Gamma(1+t)}{\Gamma(2-t)}- \frac{\Gamma(y+t)}{\Gamma(y+1-t)}\right)^{p-1}.
	\end{align*}
	Hence, it is enough to find a condition under which $S_2>S_0$. Using the functional equation of the $\Gamma$-function, we obtain
	\begin{align*}
		S_0&=\frac{\Gamma(1-s)}{\Gamma(2+s)}\left(\frac{\Gamma(1+t)}{\Gamma(2-t)}\cdot \frac{1-2t}{t}\right)^{p-1},\\
		S_2&=\frac{\Gamma(1-s)}{\Gamma(2+s)}\left(1+ \frac{(2-s)(1-s)}{(3+s)(2+s)}\right)\left(\frac{\Gamma(1+t)}{\Gamma(2-t)}\cdot \frac{1-2t}{2-t}\right)^{p-1}.
	\end{align*}
	Thus, we need to show
	\[ 1+ \frac{(2-s)(1-s)}{(3+s)(2+s)} > \left(\frac{2-t}{t}\right)^{p-1}. \]
	As $p\to 1$, the right-hand side becomes $1$ whereas the left-hand side is always larger than $1$ with equality at $s=1$ and a maximum at $s=0$. As long as $p\neq1$, the right-hand side has a singularity at $t=0$. Thus, in general, we can only assert the existence of $p_0$ and $\sigma_0$ as in the statement. Alternatively, one obtains an explicit condition when $\alpha$ is not too small. The latter follows from the equivalent inequality
	\[\alpha>\frac{4}
		{q\left(
			1+
			\left[
			1+\frac{(2-s)(1-s)}{(3+s)(2+s)}
			\right]^{\frac{1}{p-1}}
			\right)}.\]
	This finishes the proof. 
\end{proof}
Using another observation, we also obtain the following.

\begin{lemma}
	[Superharmonicity at $\pm 1$ for all $p$ and small $\sigma$]\label{lem:super_small_sigma}
	Fix $p,q>1$ and $\alpha\in  (0, \frac1q)$, and set $s=\frac{p\sigma}{2}$. If 
		\[\alpha>
	\frac{4}{q\left(1+
		\left(\frac{s^2+2}{s(2+s)}\right)^{\frac{1}{p-1}}
		\right)},\]
		then $\Delta_p^\sigma\kappa_{q,-\alpha}(\pm1) >0$.
	 In particular, there exists $\sigma_0\leq\frac{1}{p}$ such that for all $\sigma\in(0,\sigma_0)$ we have  $\Delta_p^\sigma\kappa_{q,-\alpha}(\pm1) >0$. Furthermore, if $\alpha=\frac{p-1+p\sigma}{pq}$ and $p<3-\frac{\log5}{\log3}\approx1.535$, then $\Delta_p^\sigma\kappa_{q,-\alpha}(\pm1)>0$ for any $\sigma$. 
\end{lemma}
\begin{proof}
	We proceed as in the previous proof. Observe the following. Writing $h(y) = \Gamma(y-s)/\Gamma(y+s)$, we have 
	\[ h(y) - h(y+1) = 2s \frac{\Gamma(y-s)}{\Gamma(y+1+s)}. \] 
	This observation together with Lemma~\ref{lem:Gamma} allows us to estimate $S_\infty$ from below to obtain 
	\begin{align*}
		S_\infty 
			&\geq \left(\frac{\Gamma(1+t)}{\Gamma(2-t)}- \frac{\Gamma(2+t)}{\Gamma(3-t)}\right)^{p-1} \\
				&\quad\qquad\cdot\sum_{y>2} \frac1{2s}\left(h(y-1)-h(y)+h(y+1)-h(y+2)\right) \\ 
			&= \left(\frac{\Gamma(1+t)}{\Gamma(2-t)}\cdot \frac{1-2t}{2-t}\right)^{p-1} \frac1{2s} (h(2)+h(4)) \\
			&= \frac{\Gamma(1-s)}{\Gamma(2+s)}\cdot\frac{1-s}{2s}\left(1+ \frac{(3-s)(2-s)}{(3+s)(2+s)}\right)
			\left(\frac{\Gamma(1+t)}{\Gamma(2-t)}\cdot \frac{1-2t}{2-t}\right)^{p-1}.
	\end{align*}
	Thus, we need to show
	\[ \frac{1-s}{2s}\left(1+ \frac{(3-s)(2-s)}{(3+s)(2+s)}\right) + 1+ \frac{(2-s)(1-s)}{(3+s)(2+s)} > \left(\frac{2-t}{t}\right)^{p-1} \]
	which is equivalent to 
	\[ \frac{1+s}{2s}\left(1+ \frac{(2-s)(1-s)}{(2+s)(1+s)}\right)=\frac{s^2+2}{s(2+s)}  > \left(\frac{2-t}{t}\right)^{p-1}. \]
	Equivalently, in terms of $\alpha$, 
	\[\alpha>
		\frac{4}{q\left(1+
			\left(\frac{s^2+2}{s(2+s)}\right)^{\frac{1}{p-1}}
			\right)}.\]
	Since $p,q$ and $\alpha$ are fixed, the right-hand side is bounded, but the left-hand side is unbounded as $s\to0$. Thus, the claim follows. 
	
	For the ``furthermore'' statement, notice that the left- and right-hand side are convex and decreasing in $s$. Since the derivative of the right-hand side at $s=\frac12$ is larger than that of the left-hand side, we only need to verify the inequality at $s=t=\frac12$: 
	\[ \frac95 > 3^{p-1} \iff p < 3 - \frac{\log5}{\log3}. \] 
	This finishes the proof. 
\end{proof}

\begin{lemma}
	[Superharmonicity at $\pm 1$ for large $p$ and optimal $\alpha$]\label{lem:super_optimal_plarge}
	Let $p,q>1$, $\sigma\in(0,\frac1p)$ and $\alpha=\frac{p-1+p\sigma}{pq}$. Then, there exists $p_0\geq2$ such that for all $p>p_0$ we have  $\Delta_p^\sigma\kappa_{q,-\alpha}(\pm1) >0$. 
\end{lemma}
\begin{proof}
	First, notice that superharmonicity of $\kappa_{q,-\alpha}$ at $1$ is equivalent to showing that 
	\begin{align*}
		\sum_{y>1} &
		\left(\kappa_{p,\sigma}(y-1) + \kappa_{p,\sigma}(y+1)\right) 
		\left(\kappa_{q,-\alpha}(1) - \kappa_{q,-\alpha}(y)\right)^{p-1} \\ 
		&\qquad> \kappa_{p,\sigma}(1) 
		\left(\kappa_{q,-\alpha}(0) - \kappa_{q,-\alpha}(1)\right)^{p-1}.
	\end{align*}
	Dividing the left-hand side by the right-hand side, we obtain
		\[ T:= \frac{\displaystyle
			\sum_{y>1} \left(\frac{\Gamma(y-1-s)}{\Gamma(y+s)}+ \frac{\Gamma(y+1-s)}{\Gamma(y+2+s)}\right) 
			\left(\frac{\Gamma(1+t)}{\Gamma(2-t)}- \frac{\Gamma(y+t)}{\Gamma(y+1-t)}\right)^{p-1}}
			{\displaystyle \frac{\Gamma(1-s)}{\Gamma(2+s)} \left(\frac{\Gamma(t)}{\Gamma(1-t)} - \frac{\Gamma(1+t)}{\Gamma(2-t)}\right)^{p-1}}. \] 
	Our goal is to show that $T>1$ for large enough $p$. We will start by simplifying $T$ term by term. 
	
	Using the functional equation of the $\Gamma$-function and a telescoping argument, it follows that 
	\begin{align*}
		\frac{\Gamma(1+t)}{\Gamma(2-t)}&-\frac{\Gamma(y+t)}{\Gamma(y+1-t)} \\ 
			&= \frac{\Gamma(1+t)}{\Gamma(2-t)} \left(1-\prod_{k=1}^{y-1} \frac{k+t}{k+1-t}\right) \\ 
			&= \frac{\Gamma(1+t)}{\Gamma(2-t)} \frac{\displaystyle 
				\sum_{j=1}^{y-1} (1-2t) \left(\prod_{k=1}^{j-1}(k+1-t)\right) \left(\prod_{k=j+1}^{y-1}(k+t)\right)}
				{\displaystyle\prod_{k=1}^{y-1}(k+1-t)} \\ 
			&= \frac{\Gamma(1+t)}{\Gamma(2-t)} (1-2t) \sum_{j=1}^{y-1} \frac1{j+1-t} \prod_{k=j+1}^{y-1}\frac{k+t}{k+1-t}. 
	\end{align*}
	We denote the sum by $A_y(t)$. Observe that $A_y$ is increasing and convex in $t\in(0,\frac12)$ as a sum of products of positive increasing convex functions with the same monotonicity. 
	
	Now writing 
		\[ \frac{\Gamma(t)}{\Gamma(1-t)} - \frac{\Gamma(1+t)}{\Gamma(2-t)} 
			= \frac{\Gamma(1+t)}{\Gamma(2-t)} \frac{1-2t}t, \] 
	we can simplify 
	\begin{align*}\label{eq:T}
		T &= \sum_{y>1} \frac{\Gamma(2+s)}{\Gamma(1-s)} \left(\frac{\Gamma(y-1-s)}{\Gamma(y+s)}+ \frac{\Gamma(y+1-s)}{\Gamma(y+2+s)}\right) (tA_y(t))^{p-1} \\ 
		&\geq (tA_{y_0}(t))^{p-1} \sum_{y\geq y_0} \frac{\Gamma(2+s)}{\Gamma(1-s)} \left(\frac{\Gamma(y-1-s)}{\Gamma(y+s)}+ \frac{\Gamma(y+1-s)}{\Gamma(y+2+s)}\right). \nonumber
	\end{align*}
	With $h(y) = \Gamma(y-s)/\Gamma(y+s)$, the sum evaluates to 
		\[ \frac{\Gamma(2+s)}{\Gamma(1-s)} \frac{h(y_0-1)+h(y_0+1)}{2s} 
			= \frac{1-s}{2s} \left(\prod_{k=2}^{y_0-2} \frac{k-s}{k+s} + \prod_{k=2}^{y_0} \frac{k-s}{k+s}\right). \] 
	With the same arguments as before, the latter is a convex decreasing function in $s\in(0,\frac12)$ with a singularity at $s=0$. 
	
	Observing that the choice of the optimal parameter for $\alpha$ makes $t$ an affine function of $s$, we find that $T$ is bounded from below by a positive decreasing convex function of $s$ which hence attains its minimal value at $s=\frac12$. We therefore obtain 
		\[ T\geq \left(\frac12 \sum_{j=1}^{y_0-1} \frac1{j+\frac12}\right)^{p-1} \frac12 \left(\frac{\frac32}{y_0+\frac32} + \frac{\frac32}{y_0+\frac12}\right). \]
	Choosing $y_0$ sufficiently large, the harmonic sum will exceed $2$. Hence, for $p$ large enough, the right-hand side is greater than $1$. 
\end{proof}
We use the main idea from the previous lemma to obtain a result for subharmonicity.
\begin{lemma}
	[Subharmonicity at $\pm 1$ for $p=3$, large $\sigma$ and optimal~$\alpha$]\label{lem:sub_optimal}
	Let $p=3,q>1,\sigma\in(0,\frac1p)$ and $\alpha=\frac{p-1+p\sigma}{pq}$. Then, there exists $\sigma_0\in(0,1/p)$ such that for all $\sigma\in(\sigma_0,1/p)$ we have  $\Delta_p^\sigma\kappa_{q,-\alpha}(\pm1)<0$. 
\end{lemma}
\begin{proof}
	We start with a simple observation leading to a rational bound for $\log2$. Using the series representation 
		\[ \log z = 2\sum_{k\geq0} \frac1{2k+1} 
								\left(\frac{z-1}{z+1}\right)^{2k+1} \] 
	for the logarithm valid on the positive real axis, we have 
		\[ \log2 
			\leq 2\left(\frac13 + \frac1{3^4} \sum_{k\geq0} \frac1{9^k}\right) 
			= \frac{25}{36}. \] 
	We now proceed as in the previous proof, noting that subharmonicity of $\kappa_{q,-\alpha}$ at $1$ is equivalent to showing that 
		\[ T:= \frac{\displaystyle
			\sum_{y>1} \left(\frac{\Gamma(y-1-s)}{\Gamma(y+s)}+ \frac{\Gamma(y+1-s)}{\Gamma(y+2+s)}\right) 
			\left(\frac{\Gamma(1+t)}{\Gamma(2-t)}- \frac{\Gamma(y+t)}{\Gamma(y+1-t)}\right)^2}
			{\displaystyle \frac{\Gamma(1-s)}{\Gamma(2+s)} \left(\frac{\Gamma(t)}{\Gamma(1-t)} - \frac{\Gamma(1+t)}{\Gamma(2-t)}\right)^2} \] 
	is less than 1 for large enough $\sigma$. To that end, we show that for the optimal choice of $\alpha$, the limit of $T$ as $\sigma\to\frac1p$ is less than 1. Observe that since $p=3$, the optimal choice of $\alpha$ makes $t=\frac{1+s}3$. Then, the limit $\sigma\to\frac13$ is equivalent to $s\to\frac12$ and $t\to\frac12$. With the notation from the previous proof, we have 
		\[ \lim_{t\to\frac12} tA_y(t) = \sum_{j=1}^{y-1} \frac1{2j+1} \leq \frac12 \log y \] 
	and  
	\begin{align*}
		\lim_{s\to\frac12} \frac{\Gamma(2+s)}{\Gamma(1-s)} 
				&\left(\frac{\Gamma(y-1-s)}{\Gamma(y+s)}+ \frac{\Gamma(y+1-s)}{\Gamma(y+2+s)}\right) \\
			&= \frac3{(2y-3)(2y-1)} + \frac3{(2y+1)(2y+3)} 
			\leq \frac7{4y^2}
	\end{align*}
	with the bound holding for $y\geq5$. Then, by dominated convergence 
	\begin{align*}
		\lim_{s\to\frac12} T 
			&= \sum_{y>1} 
				\left(\frac3{(2y-3)(2y-1)} + \frac3{(2y+1)(2y+3)}\right) 
				\left(\sum_{j=1}^{y-1} \frac1{2j+1}\right)^2. 
	\intertext{The first three terms can be bounded by $\frac{21}{80}$. We can thus estimate the limit further as}
			\cdots&\leq \frac{21}{80} 
				+ \frac7{16} \sum_{y\geq5} \left(\frac{\log y}y\right)^2 
			\leq \frac{21}{80} 
				+ \frac7{16} \int_4^\infty \left(\frac{\log y}y\right)^2 dy 
	\intertext{where we used the fact that the integrand is decreasing in $y$. Evaluating the elementary integral, we obtain}
			\cdots&= \frac{21}{80} 
				+ \frac7{16} \left(\log^2 2 + \log2 + \frac12\right) 
			\leq \frac{21}{80} + \frac7{16} \left(\left(\frac{25}{36}\right)^2 
						+ \frac{25}{36} + \frac12\right) \\ 
			&= \frac{103271}{103680} < 1
	\end{align*}
	where we used the observation from the beginning. 
\end{proof}
It is interesting to note that for optimal parameters, we obtain subharmonicity at $\pm 1$ for $p\approx 3$ and superharmonicity for almost all other values of $p$.

The preceding results show that it is difficult to determine the sign of our candidate for an optimal Hardy-type weight in a compact punctured set around the origin in general. Nevertheless, it would be very interesting to find explicit conditions under which $\kappa_{q,-\alpha}$ is $\Delta_p^\sigma$-superharmonic on the whole line $\ZZ$ (apart from $p=q=2$). 
	
\section{The Hardy-type Weight is Optimal}\label{s:optimal}
	
\subsection{Proof of Criticality}
	
	Recall that $p,q \in (1,\infty)$, $\sigma\in (0,\frac1p)$, as well as $\alpha\in(0,\frac1q)$. In this subsection, we also assume
	\[ p(\sigma+1-q\alpha)\geq1\quad \text{, i.e.,}\quad  \alpha\leq\frac{p-1+p\sigma}{pq}. \]
	
	In contrast to the proof of the linear case in \cite{KN}, one has to be more careful with the estimates.
	
	\begin{proof}[Proof of Theorem~\ref{thm:main}~\ref{thm:main_a}]
		In order to show criticality of the Hardy-type weight $w$, we present a sequence $(e_n)$ satisfying the properties in Lemma~\ref{lem:null-sequence}. Let $e_n\in C_c(\ZZ)$ be defined by $e_n(0)=1$ and 
			\[ e_n(x) = \left(1-\frac{\log|x|}{\log n}\right)_+ \] 
		otherwise. Clearly, $\supp(e_n) = \{x\in\ZZ\mid|x|<n\}$ and $e_n\to1$. Thus, criticality follows via Lemma~\ref{lem:null-sequence} once we have shown that $\E_{p,u}^\sigma(e_n)\to0$ where $u=\kappa_{q,-\alpha}$. Thanks to Corollary~\ref{cor:GSR}, we may pass to $\E_{p,u,1}^\sigma$ and $\E_{p,u,2}^\sigma$. We first show that $(e_n)$ is a null-sequence for $\E_{p,u,1}^\sigma$. This suffices to prove the claim for $p<2$. If $p\geq2$, we will additionally show that $(e_n)$ is a null-sequence for $\E_{p,u,2}^\sigma$ thereby proving the claim for all $p>1$. 

		We now turn our attention to $\E_{p,u,1}^\sigma(e_n)$. Using the asymptotics of $\kappa$ from Equation~(\ref{eq:kappaAsymp}) and the reverse triangle inequality, we have 
			\[ \E_{p,u,1}^\sigma(e_n) 
				\lesssim \sum_{1<y<x} (x-y)^{-1-p\sigma} (xy)^{\frac p2(-1+q\alpha)} 
					|\nabla_{x,y} e_n|^p. \] 
		Regarding the support of $e_n$, we split the sum into two parts that we denote by $S_1$ and $S_2$ corresponding to the summation over the sets $\{1<y<x<n\}$ and $\{1<y<n\leq x\}$. For $S_1$, we first estimate the summation over $x$ by an integral and substitute $t=\frac xy$. We obtain 
			\[ S_1 
				\lesssim \frac1{\log^p n} \sum_{1<y<n} y^{-p(\sigma+1-q\alpha)} 
					\int_1^\infty (t-1)^{-1-p\sigma} t^{-\frac p2 (1-q\alpha)} 
						\log^p t \,dt. \] 
		We show that the integral is finite. Since $p\sigma+\frac p2(1-q\alpha)>0$, we find $0<\ep<p\sigma+\frac p2(1-q\alpha)$. Now, let $t_\ep$ be such that $\log^p t < t^\ep$ for all $t>t_\ep$. Then, the integral can be bounded via 
			\[ \int_1^{t_\ep} (t-1)^{p-1-p\sigma} dt 
				+ \int_{t_\ep}^\infty t^{-1-p\sigma-\frac p2(1-q\alpha) + \ep} dt. \] 
		The first integral converges since $p(1-\sigma)>0$, and the second one by our choice of $\ep$. For $S_1$ we conclude that the sum converges for $p(\sigma+1-q\alpha)>1$ and has $\log$-asymptotics if $p(\sigma+1-q\alpha)=1$. In either case, $S_1$ tends to zero. 
		
		Next, we bound $S_2$. We again bound the sum over $x$ by an integral and make the change of variables $t=\frac xy$ yielding 
		\begin{align*}
			S_2 &\lesssim \frac1{\log^p n} \sum_{1<y<n} y^{-p(\sigma+1-q\alpha)} 
					\log^p \frac ny \int_{\frac ny}^\infty (t-1)^{-1-p\sigma} 
						t^{-\frac p2(1-q\alpha)} dt. 
		\intertext{If $y<\frac n2$, we bound the integral by $(\frac ny -1)^{-p\sigma-\frac p2(1-q\alpha)}$. Otherwise, it is bounded by $(\frac ny -1)^{-p\sigma}$. Furthermore, replacing the sum by an integral, we obtain} 
			\dots&\lesssim \frac{n^{-p(\sigma+1-q\alpha)}}{\log^p n} 
					\left[
						\int_1^{\frac n2} \left(\frac yn\right)^{-p(\sigma+1-q\alpha)} 
						\left(\frac ny -1\right)^{-p\sigma-\frac p2(1-q\alpha)} 
						\log^p\frac ny \,dy 
					\right. \\ 
					&\qquad\qquad\qquad\qquad+ 
					\left.
						\int_{\frac n2}^n \left(\frac yn\right)^{-p(\sigma+1-q\alpha)} 
						\left(\frac ny -1\right)^{-p\sigma} \log^p\frac ny \,dy 
					\right]. 
		\intertext{We now substitute $e^{-u} = \frac yn$ and we get} 
			\dots&= \frac{n^{1-p(\sigma+1-q\alpha)}}{\log^p n} 
					\left[
						\int_{\log 2}^{\log n} e^{(\frac p2(1-q\alpha)-1)u} 
						(1-e^{-u})^{-p\sigma-\frac p2(1-q\alpha)} u^p du 
					\right. \\ 
					&\qquad\qquad\qquad\qquad+ 
					\left.
						\int_0^{\log 2} e^{(p(1-q\alpha)-1)u} 
						\left(1-e^{-u}\right)^{-p\sigma} u^p du 
					\right] \\ 
				&\lesssim \frac{n^{1-p(\sigma+1-q\alpha)}}{\log^p n} 
					\left[
						\int_{\log 2}^{\log n} e^{(\frac p2(1-q\alpha)-1)u} u^p du 
					+
						\int_0^{\log 2} u^{p(1-\sigma)} du 
					\right]. 
		\intertext{The second integral is a constant. The first integral, depending on the sign of $\frac p2(1-q\alpha)-1$, is a constant or, via a supremum estimate, bounded by $n^{\frac p2(1-q\alpha)-1} \log^{p+1} n$. Thus,} 
			\dots&\lesssim \frac{n^{1-p(\sigma+1-q\alpha)}}{\log^p n} 
					\left[1 + n^{\frac p2(1-q\alpha)-1} \log^{p+1} n\right] \\ 
				&=\frac{n^{1-p(\sigma+1-q\alpha)}}{\log^p n} 
					+ n^{-\frac p2(2\sigma+1-q\alpha)} \log n 
		\end{align*}
		which tends to zero. 

		We now turn to $\E_{p,u,2}^\sigma$. In the following, we assume that $p\geq2$. We proceed in a similar manner as before. Using the asymptotics of $\kappa$ and the reverse triangle inequality, $\E_{p,u,2}^\sigma(e_n)$ is bounded by 
			\[ \sum_{1<y<x} (x-y)^{-1-p\sigma} (xy)^{-1+q\alpha} 
				\left(y^{-1+q\alpha} - x^{-1+q\alpha}\right)^{p-2} 
				(e_n(x)+e_n(y))^{p-2} |\nabla_{x,y} e_n|^2. \] 
		As before, we split this sum into two terms, $T_1$ and $T_2$, corresponding to the summation over the sets $\{1<y<x<n\}$ and $\{1<y<n\leq x\}$ respectively. We will first handle $T_1$. From $e_n\leq1$, we get 
			\[ T_1 \lesssim \sum_{1<y<x<n} (x-y)^{-1-p\sigma} (xy)^{-1+q\alpha} 
				\left(y^{-1+q\alpha} - x^{-1+q\alpha}\right)^{p-2} 
				|\nabla_{x,y} e_n|^2 \] 
		which we will split further into two parts, $T_{1,1}$ and $T_{1,2}$, according to whether $x>cy$ or $x<cy$ for some constant $c\gg1$ determined as follows. Since $p\sigma+1-q\alpha>0$, we can find $\ep>0$ such that $p\sigma+1-q\alpha-\ep>0$. Choose $c$ such that $\log^2 t < t^\ep$ for $t>c$. Let us now turn to $T_{1,1}$. Since $x>cy$, we can estimate 
		\begin{align*}
			T_{1,1} &\lesssim \sum_{1<y<cy<x<n} (x-y)^{-1-p\sigma} (xy)^{-1+q\alpha} 
						y^{-(p-2)(1-q\alpha)} |\nabla_{x,y} e_n|^2 \\ 
					&\lesssim \frac1{\log^2 n} \sum_{y<\frac nc} y^{-p(\sigma+1-q\alpha)} 
						\int_c^\infty t^{-(1-q\alpha)} (t-1)^{-1-p\sigma} \log^2 t \,dt 
		\end{align*}
		where we substituted $t=\frac xy$. By the choice of $\ep$ and $c$, the integrand is bounded by $t^{-(p\sigma+1-q\alpha-\ep)-1}$. Thus, the integral is finite. We hence obtain that the sum is bounded for $p(\sigma+1-q\alpha)>1$ and has $\log$-asymptotics if $p(\sigma+1-q\alpha)=1$. In either case, the product with $\frac1{\log^2 n}$ forms a null-sequence and $T_{1,1}$ tends to zero. 
		
		We now turn to $T_{1,2}$. We use the mean value theorem to estimate 
			\[ \left(\frac1y\right)^{1-q\alpha} - \left(\frac1x\right)^{1-q\alpha} 
				\lesssim x^{q\alpha}\left(\frac1y - \frac1x\right) 
				= x^{-1+q\alpha} \left(\frac xy -1\right). \] 
		Together with the fact that $e_n\leq1$, we have 
		\begin{align*}
			T_{1,2} &\lesssim \sum_{\substack{1<y<x<n\\x<cy}} y^{1-p(\sigma+1-q\alpha)} 
						\left(\frac xy\right)^{-(p-1)(1-q\alpha)} 
						\left(\frac xy -1\right)^{p(1-\sigma)-3} \log^2 \frac xy \\ 
					&\lesssim \frac1{\log^2 n} \sum_{1<y<n} y^{-p(\sigma+1-q\alpha)} 
						\int_1^c t^{-(p-1)(1-q\alpha)} (t-1)^{p(1-\sigma)-3} 
						\log^2 t \,dt \\ 
					&\lesssim \frac1{\log^2 n} \sum_{1<y<n} y^{-p(\sigma+1-q\alpha)} 
						\int_1^c (t-1)^{p(1-\sigma)-1} dt. 
		\end{align*}
		The integral is again finite since $p(1-\sigma)>0$ and we see that $T_{1,2}$ tends to zero. 
		
		Next, we consider $T_2$. Since $x>n$, we have 
			\[ T_2 \lesssim \sum_{1<y<n<x} (x-y)^{-1-p\sigma} (xy)^{-1+q\alpha} 
				\left(y^{-1+q\alpha} - x^{-1+q\alpha}\right)^{p-2} e_n(y)^p. \] 
		As we have done before, we will split the sum into two terms, $T_{2,1}$ and $T_{2,2}$, this time depending on whether $x>2y$ or $x<2y$. For $T_{2,1}$, we have 
		\begin{align*}
			T_{2,1} &\lesssim \sum_{\substack{1<y<n<x\\x>2y}} (x-y)^{-1-p\sigma} 
						(xy)^{-1+q\alpha} y^{-(p-2)(1-q\alpha)} e_n(y)^p \\ 
					&\lesssim \frac1{\log^p n} \sum_{1<y<n} y^{-p(\sigma+1-q\alpha)} 
						\log^p \frac ny \int_{\frac ny\vee2}^\infty t^{-(1-q\alpha)} 
						(t-1)^{-1-p\sigma} dt 
		\intertext{where we substituted $t=\frac xy$. For $y<\frac n2$, we bound the integral by $(\frac ny -1)^{-p(\sigma+1-q\alpha)}$ and otherwise by a constant. We get} 
				\dots&\lesssim \frac1{\log^p n} 
					\left[
						\sum_{1<y<\frac n2} (n-y)^{-p(\sigma+1-q\alpha)} \log^p \frac ny 
					\right. \\ 
				&\qquad\qquad\qquad\qquad+ 
					\left.
						\sum_{\frac n2<y<n} y^{-p(\sigma+1-q\alpha)} \log^p \frac ny
					\right] \\ 
				&\lesssim\frac{n^{-p(\sigma+1-q\alpha)}}{\log^p n} 
					\left[
						\int_1^{\frac n2} \left(1-\frac yn\right)^{-p(\sigma+1-q\alpha)} 
						\log^p \frac ny \,dy 
					\right. \\ 
				&\qquad\qquad\qquad\qquad+ 
					\left.
						\int_{\frac n2}^n \left(\frac yn\right)^{-p(\sigma+1-q\alpha)} 
						\log^p \frac ny \,dy 
					\right]. 
		\intertext{We make the substitution $e^{-u} = \frac yn$ and obtain} 
				\dots&= \frac{n^{1-p(\sigma+1-q\alpha)}}{\log^p n} 
					\left[
						\int_{\log 2}^{\log n} e^{-u} (1-e^{-u})^{-p(\sigma+1-q\alpha)} 
						u^p du 
					\right. \\ 
				&\qquad\qquad\qquad\qquad+ 
					\left.
						\int_0^{\log 2} e^{(p(\sigma+1-q\alpha)-1)u} u^p du 
					\right] \\ 
				&\lesssim \frac{n^{1-p(\sigma+1-q\alpha)}}{\log^p n} 
					\left[
						\int_{\log 2}^{\log n} e^{-u} u^p du 
					+
						\int_0^{\log 2} u^p du 
					\right]. 
		\end{align*}
		Since both integrals are finite, we obtain that $T_{2,1}$ tends to zero. 
		
		We are now left with $T_{2,2}$. Since $x<2y$, we again use the mean value theorem yielding 
		\begin{align*}
			T_{2,2} &\lesssim \sum_{1<y<n<x<2y} (x-y)^{-1-p\sigma} x^{-(p-1)(1-q\alpha)} 
						y^{-(1-q\alpha)} \left(\frac xy -1\right)^{p-2} e_n(y)^p \\ 
					&\lesssim\frac1{\log^p n} \sum_{\frac n2<y<n} 
						y^{-p(\sigma+1-q\alpha)} 
						\log^p \frac ny \int_{\frac ny}^2 t^{-(p-1)(1-q\alpha)} 
						(t-1)^{p(1-\sigma)-3} dt \\ 
					&\lesssim \frac{n^{-p(\sigma+1-q\alpha)}}{\log^p n} 
						\int_{\frac n2}^n \left(\frac yn\right)^{-p(\sigma+1-q\alpha)} 
						\left(\frac ny -1\right)^{p(1-\sigma)-2} 
						\log^p \frac ny \,dy. 
		\intertext{We make the substitution $e^{-u} = \frac yn$ and obtain} 
				\dots&= \frac{n^{1-p(\sigma+1-q\alpha)}}{\log^p n} 
						\int_0^{\log 2} e^{(p(1-q\alpha)+p-3)u} 
						(1-e^{-u})^{p(1-\sigma)-2} u^p du \\ 
					&\lesssim \frac{n^{1-p(\sigma+1-q\alpha)}}{\log^p n} 
						\int_0^{\log 2} u^{p(1-\sigma)+p-2} du. 
		\end{align*}
		The integral is finite. Hence $T_{2,2}$ also tends to zero, completing the proof. 
	\end{proof}
	
\subsection{Proof of Null-criticality}
Recall that $p,q \in (1,\infty)$, $\sigma\in (0,\frac1p)$, as well as $\alpha\in(0,\frac1q)$. In this subsection, we assume also
\[ p(\sigma+1-q\alpha)\leq1\quad \text{, i.e.,}\quad  \alpha\geq\frac{p-1+p\sigma}{pq}. \]

In contrast to the proof of the linear case in \cite{KN} (and the absence of the spectral theorem), the asymptotic behaviour of $w$ is not straightforward. 

\begin{proof}[Proof of Theorem~\ref{thm:main}~\ref{thm:main_b}]
	We have to show that $\kappa_{q,-\alpha}\notin \ell^p(\ZZ,\abs{w})$, where $w=\Delta_p^\sigma \kappa_{q,-\alpha}/ \kappa^{p-1}_{q,-\alpha}$. 	
	
	We have the following asymptotic behaviour:
	\begin{align*}
		w(x)\asymp \abs{x}^{-p\sigma},\quad  \text{ and } \quad  \kappa_{q,-\alpha}(x)\asymp \abs{x}^{-1+q\alpha}, \quad x\to \infty.
	\end{align*}
	Note that the large-scale behaviour of $\kappa_{q,-\alpha}$ follows from Equation~\eqref{eq:kappaAsymp}, whereas that of $w$ is shown in Proposition~\ref{prop:optimal constant}. 
	
	Hence, outside of a compact set (where $w$ is positive in any case by Proposition~\ref{prop:superharmonic}), the relevant series is, up to multiplicative constants, an infinite sum over $\abs{x}^{-p\sigma+ p(-1+q\alpha)}$, $x\in \ZZ$, which is not finite if and only if $p(\sigma+1-q\alpha)\leq1$.
\end{proof}

\section{On the Optimality of the Constant} \label{s:constant}

In this section, we compute the constant at infinity. More precisely, we have the following proposition. 

\begin{proposition}[Hardy constant at infinity] \label{prop:optimal constant}
	Let $p,q>1$, $\sigma\in(0,\frac1p)$ and $\alpha\in(0\vee\frac{p-2}{(p-1)q},\frac1q)$. The constant of the Hardy-type weight 
		\[ w = \frac{\Delta_p^\sigma \kappa_{q,-\alpha}}{\kappa_{q,-\alpha}^{p-1}} \] 
	at infinity is given by 
		\[ C_{\infty}:=c_{p,\sigma} \int_0^\infty t^{-1-p\sigma} \left[
				\left(1-(1+t)^{-1+q\alpha}\right)^{p-1} 
				+ \left(1-|1-t|^{-1+q\alpha}\right)^{\langle p-1\rangle} \right]dt, \]
				i.e.,
				 \[ w(x)= {C_{\infty}}{x^{-p\sigma}}+ O(x^{-1-p\sigma}), \qquad \abs{x}\to \infty.\]
\end{proposition}

\begin{proof}
	Let $u=\kappa_{q,-\alpha}$. We have 
	\begin{align*}
		w(x) 
			&= \frac{\Delta_p^\sigma u(x)}{u(x)^{p-1}} \\ 
			&= \frac1{x^{p\sigma}} 
				\sum_{y\geq1} x^{p\sigma} \kappa_{p,\sigma}(y) \left[
				\left(1-\frac{u(x+y)}{u(x)}\right)^{p-1} 
				+ \left(1-\frac{u(x-y)}{u(x)}\right)^{\langle p-1\rangle}\right]. 
	\end{align*}
	We define the step function $g_x$ on $(0,\infty)$ as follows. On each interval 
		\[ I_y := \left[\frac yx,\frac{y+1}x\right), \]
	set 
		\[ g_x(t) 
			:= x^{1+p\sigma} \kappa_{p,\sigma}(y) \left[
				\left(1-\frac{u(x+y)}{u(x)}\right)^{p-1} 
				+ \left(1-\frac{u(x-y)}{u(x)}\right)^{\langle p-1\rangle}
				\right]. \]	
	Then, by construction the sum above is equal to $\int_0^\infty g_x(t)dt$. It thus suffices to identify the $L^1$-limit of $g_x$. Fix a compact set contained in $(0,\infty)\setminus\{1\}$. If $t\in I_y$, then $y/x\to t$ as $x\to \infty$ uniformly on such a compact set. Hence, by Equation~(\ref{eq:kappaAsymp}), 
		\[ x^{1+p\sigma} \kappa_{p,\sigma}(y) \to c_{p,\sigma} t^{-1-p\sigma} \]
	uniformly there. Similarly, 
		\[ \frac{u(x+y)}{u(x)} \to (1+t)^{-1+q\alpha}, \] 
	and, using the evenness of $u$, 
		\[ \frac{u(x-y)}{u(x)} \to |1-t|^{-1+q\alpha}. \] 
	Therefore, $g_x$ converges locally uniformly on $(0,\infty)\setminus\{1\}$ to 
		\[ g(t) := c_{p,\sigma} t^{-1-p\sigma} \left[
				\left(1-(1+t)^{-1+q\alpha}\right)^{p-1} 
				+ \left(1-|1-t|^{-1+q\alpha}\right)^{\langle p-1\rangle} \right]. \] 
	It remains to check that no mass is lost near $0$, near $1$, or at infinity. 
	
	Near $0$, observe that by Equation~(\ref{eq:kappaAsymp}), 
		\[ 1-\frac{u(x+y)}{u(x)} = (1-q\alpha)\frac yx + O\left(\frac{y^2}{x^2}\right) \] 
	and similarly 
		\[ 1-\frac{u(x-y)}{u(x)} = -(1-q\alpha)\frac yx + O\left(\frac{y^2}{x^2}\right). \]
	Now, by oddness of the French power $(\cdot)^{\langle p-1\rangle}$, the leading-order terms cancel and we have 
		\[ \left(1-\frac{u(x+y)}{u(x)}\right)^{p-1} 
			+ \left(1-\frac{u(x-y)}{u(x)}\right)^{\langle p-1\rangle} 
			= 	O\left(\left(\frac yx\right)^p\right). \] 
	It follows that, uniformly in $x$, 
		\[ \int_0^\delta |g_x(t)|dt \lesssim \delta^{p-p\sigma}, \] 
	which tends to zero as $\delta\to0$. 
	
	Near $t=1$, the only potentially singular term is the one involving $u(x-y)$. Due to Equation~(\ref{eq:kappaAsymp}), 
		\[ \frac{u(x-y)}{u(x)} 
			\lesssim \left(\frac{1+|x-y|}x\right)^{-1+q\alpha}. \] 
	Then, uniformly in $x$, we have that 
		\[ \int_{1-\delta}^{1+\delta} |g_x(t)|dt 
			\lesssim \int_{1-\delta}^{1+\delta} \left(1+|1-t|^{(-1+q\alpha)(p-1)}\right)dt 
			\to 0 \] 
	as $\delta\to0$ if $(-1+q\alpha)(p-1)>-1$, i.e., $\alpha>\frac{p-2}{(p-1)q}$. 
	
	Finally, for large $t$, the terms involving $u$ remain bounded, while the kernel $\kappa_{p,\sigma}$ provides the decay $t^{-1-p\sigma}$. Hence, 
		\[ \sup_x \int_M^\infty |g_x(t)| dt \lesssim M^{-p\sigma}, \] 
	which tends to zero as $M\to\infty$. 
	
	We have therefore shown that $g_x$ converges to $g$ uniformly on compact sets away from $0$ and $1$, and that the $L^1$-mass of $g_x$ near $0$, near $1$ and at infinity can be made uniformly arbitrarily small. Hence, $g_x$ converges to $g$ in $L^1$ and the claim follows. 
\end{proof}

Introducing a cut-off $\ep>0$ at the origin, splitting the defining integral for $C_\infty$ at 1, and applying the changes of variables $y=1\pm t$ and $y=t-1$, followed by the inversion $y\mapsto y^{-1}$ on the resulting integrals over $(1,\infty)$, we obtain, after recombining the terms and letting $\ep\to0$, 
	\[ C_\infty = c_{p,\sigma} I, \] 
where $I$ is the integral appearing in the proof of Proposition~\ref{prop:superharmonic}. 

\section{A Strictly Positive Hardy Weight}\label{s:positive weight}

In this section, we prove Theorem~\ref{thm:strictly positive weight}. We will show that $u=\kappa_{q,-\alpha}\wedge\ep$ is superharmonic for an appropriate choice of $\ep$ and deduce that $\frac{\Delta_p^\sigma u}{u^{p-1}}$ is an optimal Hardy weight for~$\Delta_p^\sigma$. 

\begin{proof}[Proof of Theorem~\ref{thm:strictly positive weight}]
	By Proposition~\ref{prop:superharmonic}, there exists a finite subset $K$ of $\ZZ$ such that $\kappa_{q,-\alpha}$ is superharmonic outside of $K$. We may assume $K$ to be symmetric. Now, by the monotonicity shown in Lemma~\ref{lem:Gamma}, there exists $\ep>0$ such that $\kappa_{q,-\alpha}<\ep$ outside of $K$ and $\kappa_{q,-\alpha}>\ep$ inside of $K$. We then define $u:=\kappa_{q,-\alpha}\wedge\ep$ and split its Laplacian as follows 
		\[ \Delta_p^\sigma u(x) 
			= \left(\sum_{y\in K} + \sum_{y\notin K}\right) \kappa_{p,\sigma}(x-y) (u(x)-u(y))^{\langle p-1\rangle}. \]
	Observe that $u$ is constant on $K$ and strictly smaller outside where it coincides with $\kappa_{q,-\alpha}$. Thus, for $x\in K$, the first sum vanishes while the second sum consists of strictly positive terms. For $x\notin K$, the second sum coincides with the respective sum for $\kappa_{q,-\alpha}$ while the first sum is strictly larger than its respective sum for $\kappa_{q,-\alpha}$. Since $\kappa_{q,-\alpha}$ is superharmonic on $\ZZ\setminus K$ by Proposition~\ref{prop:superharmonic}, we find that $\Delta_p^\sigma u(x) > 0$. Hence $w:=\Delta_p^\sigma u/u^{p-1}$ is strictly positive everywhere. By Theorem~\ref{thm:GSR}, $w$ is a Hardy weight. Furthermore, observe that since $u\leq\kappa_{q,-\alpha}$, we have $\E_{p,u,1}^\sigma\leq\E_{p,\kappa_{q,-\alpha},1}^\sigma$ and $\E_{p,u,2}^\sigma\leq\E_{p,\kappa_{q,-\alpha},2}^\sigma$. Hence, combining Theorem~\ref{thm:main}, Corollary~\ref{cor:GSR} and Corollary~\ref{cor:null}, $w$ is critical. Finiteness of $K$ implies that $w$ has the same asymptotics as the weight in Theorem~\ref{thm:main}. Then in particular $u\notin\ell^p(\ZZ,w)$. This shows that $w$ is an optimal Hardy weight. 
\end{proof}
	
\appendix
\section{$p$-Null-Criticality implies $p$-Optimality at Infinity}
	For general locally summable graphs, it seems to be folklore in the case $p=2$ that $p$-null-criticality implies $p$-optimality at infinity, and recently it has been proved explicitly in \cite[Lemma~2.12]{Hake} and \cite[Proposition~15]{HKPLattice}. Both proofs follow the argument from \cite[Lemma~6]{KN} on $\ZZ$. 
	
	Here we want to show the quasi-linear analogue of the results in \cite{Hake,HKPLattice, KN} on locally summable graphs. Note that it also generalizes \cite[Theorem~2.6]{F} since we do not assume that the Agmon ground state is of bounded oscillation. This distinction is crucial since the Agmon ground state which we constructed for the fractional $p$-Laplacian is not of bounded oscillation. The main idea is to make use of the non-linear ground state representation formula once more. 
	
	We say that the Hardy weight $w$ is \emph{$p$-optimal at infinity} if 
	\[ \E_p(\phi)\geq (1+\lambda)w_p(\phi), \qquad \phi \in C_c(X\setminus K), \]
	for some finite $K\sse X$ implies $\lambda \leq 0$. This implies that the constant appearing in the Hardy-type weight is the best possible.
	
	\begin{theorem}
		Let $b$ be a graph over $(X,m)$ and $p\in (1,\infty)$. Assume that $w=\Delta_pu/u^{p-1}$ is a $p$-critical Hardy weight with Agmon ground state $u$. If for some finite $K\sse X$ and $\tilde{w} \geq 0$, we have 
		\[ \E_p(\phi)\geq (w+\tilde{w})_p(\phi), \qquad \phi \in C_c(X\setminus K), \]
		then $u\in \ell^p(X,\tilde{w}m)$.
		
		In particular, if $w$ is $p$-null-critical then it is $p$-optimal at infinity.
	\end{theorem}
	\begin{proof}
		Let $\eta\in C_c(X)$ be such that $0\leq \eta \leq 1$, and $K\sse X$ finite. Let $\phi=\eta1_{X\setminus K}$. Then by assumption and the general ground state representation formula from Section~\ref{s:GSR},
		\begin{align*}
		 \tilde{w}_p(u\phi) 
		 	\leq \E_p(u\phi)- w_p(u\phi) 
		 	\asymp \E_{p,u}(\phi).
		\end{align*}
		We now show that the latter is bounded by $C_p(\E_{p,u}(\eta) + \E_{p,u}(1_K))$. To that end, we write $\E_{p,u}(\phi)$ as 
		\begin{align*}
			\Bigg(\sum_{x,y\notin K}+&2\sum_{\substack{x\notin K\\y\in K}}\Bigg) 
			b(x,y) u(x)u(y)(\nabla_{x,y}\phi)^{2} \\
			&\cdot\left( \bigl(u(x)u(y)\bigr)^{1/2}\abs{\nabla_{x,y}\phi}+ \frac{\abs{\phi(x)}+ \abs{\phi(y)}}{2}\abs{\nabla_{x,y}u} \right)^{p-2},
		\end{align*}
		since $\phi=\eta1_{X\setminus K}$ vanishes on $K$. For the first sum, we have
		\begin{align*}
			&\sum_{x,y\notin K}b(x,y) u(x)u(y)(\nabla_{x,y}\eta)^{2} \\
			&\qquad \cdot\left( \bigl(u(x)u(y)\bigr)^{1/2}\abs{\nabla_{x,y}\eta}+ \frac{\abs{\eta(x)}+ \abs{\eta(y)}}{2}\abs{\nabla_{x,y}u} \right)^{p-2},
		\end{align*}
		which is bounded by $\E_{p,u}(\eta)$. In the second sum, the summands reduce to 
		\begin{align*}
			b(x,y)u(x)u(y)(\eta (x))^{2}\left( \bigl(u(x)u(y)\bigr)^{1/2}\abs{\eta (x)}+ \frac{\abs{\eta (x)}}{2}\abs{\nabla_{x,y}u} \right)^{p-2},
		\end{align*}
		which vanishes whenever $\eta (x)=0$. Thus, we can equivalently sum over $x\in\supp(\eta)\setminus K$ and $y\in K$. Here, since $\eta\in C_c(X)$, there exists a constant $c>0$ such that $1 \geq \eta(x) \geq c$ for all $x\in\supp(\eta)\setminus K$, i.e., $\eta \asymp 1_{X\setminus K}$ on $\supp(\eta)\setminus K$. Hence, using $\abs{\nabla_{x,y}1_{X\setminus K}}= \abs{\nabla_{x,y}1_{K}}$ for all $x,y\in X$, we get for all $x\in\supp(\eta)\setminus K$ and $y\in K$,
		\begin{multline*}
			(\eta (x))^{2}\left( \bigl(u(x)u(y)\bigr)^{1/2}\abs{\eta (x)}+ \frac{\abs{\eta (x)}}{2}\abs{\nabla_{x,y}u} \right)^{p-2} \\
			\asymp (\nabla_{x,y}1_{K})^{2}\left( \bigl(u(x)u(y)\bigr)^{1/2}\abs{\nabla_{x,y} 1_{ K}}+ \frac{\abs{1_{X\setminus K} (x)}}{2}\abs{\nabla_{x,y}u} \right)^{p-2}.
		\end{multline*}
		Replacing $|1_{X\setminus K}(x)|$ by $|1_K(x)|+|1_K(y)|$ and summing over $x,y\in X$, we obtain $\E_{p,u}(1_K)$ as an upper bound. 
		
		Now, we use the $p$-criticality of $\E_p$ to conclude the proof. By Lemma~\ref{lem:null-sequence} there is an increasing null-sequence $(e_n)$ in $C_c(X)$ such that $e_n \to u$ pointwise and $\E_p(e_n)-w_p(e_n)\to 0$. Set $\eta_n= e_n/u \in C_c(X)$. Then, $0\leq \eta_n \nearrow 1$. Thus, we can replace $\eta$ by $\eta_n$ in the preceding calculations and obtain
		\[\E_{p,u}(\eta_n)\asymp \E_p(e_n)-w_p(e_n)\to 0.\]
		Combining these estimates with Fatou's lemma and $\phi_n:= \eta_n 1_{X\setminus K}$,
		\begin{align*}
			\tilde{w}_p(u 1_{X\setminus K}) \leq \liminf_n\tilde{w}_p(u\phi_n)\leq C_p \liminf_n (\E_{p,u}(\eta_n) + \E_{p,u}(1_K)) < \infty.
		\end{align*} 
		Since $K$ is finite, we have $u\in \ell^p(X, \tilde{w}m)$.
	\end{proof}
	
	We remark that the proof remains valid if, instead of the $p$-Laplacian, we consider a $p$-Schrödinger operator provided the corresponding $p$-energy functional is non-negative.
	
	\section{The situation for the restricted fractional $p$-Laplacian}
	
	As explained in Remark~\ref{r:fracpLapl}, there are other fractional $p$-Laplacians, and motivated by the recent work \cite{Dyda} on Hardy weights with respect to the restricted fractional $p$-Laplacian, we show briefly how to obtain optimal weights in this setting. This $p$-Laplacian is associated with the graph $b_{p,\sigma}$ over $(\ZZ,1)$ with weights $b_{p,\sigma}(x,y):=\abs{x-y}^{-1-p\sigma}$, $x,y\in \ZZ$. To distinguish the spectral fractional $p$-Laplacian from the $p$-energy functional of the main part of this article, let us denote the corresponding $p$-Laplacian and $p$-energy functional by $\LL_p^\sigma$ and $\EE_p^\sigma$, respectively. In the linear case $p=2$, this operator can be associated with $\alpha$-stable Lévy processes, see e.g. \cite{BD, BDL}.
	
	We note that this operator, defined in terms of powers of the distance function, is much simpler to handle than the one from the main part which comes from certain ratios of $\Gamma$-functions.
	
	\begin{theorem}\label{thm:rfl}
	Let $p,q>1$, $\sigma\in (0,\frac1p)$ and $\alpha\in (0,\frac1q)$. Let $b_{p,\sigma}$ be a graph over $(\ZZ, 1)$ defined above. Define $u_0\in C(\ZZ)$ via $u_0(0)=0$ and $u_0(x)=\abs{x}^{-1+q\alpha}$ for $x\neq 0$, as well as 
	\[ w_0=\frac{\LL_p^\sigma u_0}{u_0^{p-1}}, \quad \text{on } \ZZ\setminus \set{0}. \]
	Define $u_\gamma\in C(\ZZ)$ via $u_\gamma(0)=\gamma$ with $\gamma \in [1,\infty)$ and $u_\gamma(x)=\abs{x}^{-1+q\alpha}$ for $x\neq 0$, as well as 
	\[ w_\gamma=\frac{\LL_p^\sigma u_\gamma}{u_\gamma^{p-1}} \quad \text{on } \ZZ. \]
	
	\begin{enumerate}[label = ({\alph*})]	
		\item\label{thm:rfl_a} If $p(\sigma+1-q\alpha)\geq1$, then $w_0$ and $w_\gamma$ are $p$-critical with respect to $\EE_p^\sigma$ on $C_c(\ZZ\setminus \set{0})$ and $C_c(\ZZ)$, resp.
		\item\label{thm:rft_b} If  $p(\sigma+1-q\alpha)\leq1$, then $u_\beta\notin \ell^p(\ZZ\setminus \set{0},\abs{w_\beta})$ for $\beta\in\set{0,\gamma}$.
	\end{enumerate}
	In particular, if $p(\sigma+1-q\alpha)=1$, then $w_0$ and $w_\gamma$ are optimal Hardy-type weights on $\ZZ\setminus \set{0}$ and $\ZZ$, resp.
	
	In the following assume that $\alpha>\frac{p-2+p\sigma}{(p-1)q}$.
	
	Then, $w_\beta>0$ outside of a compact set, and an asymptotic analysis yields $w_\beta(x)\gtrsim \abs{x}^{-p\sigma}$ as $\abs{x}\to \infty$ for $\beta\in\set{0,\gamma}$. 
	
	Furthermore, $\LL_p^\sigma u_0(0)<0$ and $\LL_p^\sigma u_0>0$ on $\ZZ\setminus \set{0}$, i.e., $w_0$ is a Hardy weight on $\ZZ\setminus \set{0}$.
	
	 We also have $w_{\gamma}(0)> 0$ and $w_\gamma(\pm 1)> 0$ if $\gamma\in \set{0}\cup [1,1+C^{\frac{1}{p-1}})$, where $C$ is a positive constant depending on all parameters explicitly given in the proof.
	\end{theorem}
	\begin{proof}
		The proofs of statements \ref{thm:rfl_a} and \ref{thm:rft_b} follow by the same arguments as in Section~\ref{s:optimal}. The positivity of $w_\beta$, $\beta\in \set{0,\gamma}$, outside of a compact set can also be proved as in Proposition~\ref{prop:superharmonic}. In fact, this proof shows (adapted to the present setting) that we even have $w_0>0$ outside of $\set{0, \pm 1}$ (note that $\LL_p^\sigma u_0(x)\geq x^{-p\sigma +(p-1)(-1+q\alpha)}I + \kappa_{p,\sigma}(x)u_0^{p-1}(x) >0 $, where $I$ is defined as in the proof of Proposition~\ref{prop:superharmonic}). 
		
		A direct computation shows that $\LL_p^\sigma u_0(0)<0$, $\LL_p^\sigma u_0(\pm 1)>0$ and $w_\gamma(0) > 0$. 
		
		Consider $\LL_p^\sigma u_\gamma(1)=T_{\infty}-T_0$, where $T_0=(\gamma-1)^{p-1}$, and 
		\[ T_{\infty}=\sum_{y>1}\left(\frac{1}{(y-1)^{1+p\sigma}}+ \frac{1}{(y+1)^{1+p\sigma}}\right)\left(1- \frac{1}{y^{1-q\alpha}}\right)^{p-1}. \]
		Then,
		\begin{align*}
			T_\infty 
				&\geq \sum_{y>1}\left(1+ \frac{1}{(1+\frac1y)^{1+p\sigma}}\right)\left(y^{1-q\alpha}-1\right)^{p-1} y^{1-p\sigma-(p-1)(1-q\alpha)} \\ 
				&\geq \left(1+\left(\tfrac23\right)^{1+p\sigma}\right)\left(2^{1-q\alpha}-1\right)^{p-1} \sum_{y>1} y^{1-p\sigma-(p-1)(1-q\alpha)} \\ 
				&\geq \left(1+\left(\tfrac23\right)^{1+p\sigma}\right)\left(2^{1-q\alpha}-1\right)^{p-1} \frac{2^{-p\sigma-(p-1)(1-q\alpha)}}{p\sigma+(p-1)(1-q\alpha)}
				=:C.
		\end{align*}	
	Hence, $T_\infty\geq T_0$ if $C\geq (\gamma -1 )^{p-1}$, i.e., $\gamma \in [1,1+ C^{\frac{1}{p-1}}]$.
	\end{proof}
	
	Moreover, the corresponding constant at infinity can also be obtained as in Proposition~\ref{prop:optimal constant}, that is, for $\beta\in \set{0,\gamma}$
	\[ w_\beta(x)= \frac{C_{\infty}}{c_{p,\sigma}}x^{-p\sigma}+ O(x^{-1-p\sigma}), \qquad \abs{x}\to \infty.\]

	\section{Alternative proof of Proposition~\ref{prop:subharmonicat1}}
	
	We close this paper with the following observation. It gives alternative bounds on $p,q, \alpha$ and $\sigma$.

	\begin{proof}[Alternative proof of Proposition~\ref{prop:subharmonicat1}]
		Recall the following basic consequence of Gautschi's inequality: For all $y> 0$ and $s,t\in (0,1/2)$ with $y-t>0$ and $y-s > 0$, we have
		\begin{align}
			\frac{1}{({y}+t)^{1-2t}} &\leq\frac{\Gamma({y}+t)}{\Gamma({y}+1-t)}\leq \frac{1}{({y}-t)^{1-2t}}, \label{eq:Gautschi_t}\\
			\frac{1}{({y}+s)({y}-s)^{2s}} &\leq\frac{\Gamma({y}-s)}{\Gamma({y}+1+s)}\leq \frac{1}{({y}-s)({y}+s)^{2s}}.\label{eq:Gautschi_s}
		\end{align}
		
		It suffices to consider $x=1$. We use the same notation as in the proof of Lemma~\ref{lem:superharmonicat1_p}.
			
		We first estimate $S_{\infty}$ from above. We focus on the $t$-terms first. Using \eqref{eq:Gautschi_t},
		\begin{align*}
			\frac{\Gamma(1+t)}{\Gamma(2-t)}- \frac{\Gamma(y+t)}{\Gamma(y+1-t)} \leq  \frac{\Gamma(1+t)}{\Gamma(2-t)}\left( 1- \left(\frac{1-t}{y+t}\right)^{1-2t}\right).
		\end{align*}
		Consider the function $f\colon (0,\infty)\to (0,\infty), f(x)=x^{1-2t}$. Note that $r:=\frac{1-t}{y+t}<1$. By the mean value theorem,
		\begin{multline*}
			f(1)-f(r)\leq (1-2t)r^{-2t}(1-r)=(1-2t)\left(\frac{y+t}{1-t}\right)^{2t}\frac{y-(1-2t)}{y+t}\\
			\leq \frac{(1-2t)y^{2t}}{(1-t)^{2t}}\cdot \frac{1-\frac{1-2t}{y}}{(1+\frac{t}{y})^{1-2t}}\leq \frac{(1-2t)y^{2t}}{(1-t)^{2t}}.
		\end{multline*}
		
		In particular, we can factor out $\left((1-2t)\frac{\Gamma(1+t)}{\Gamma(2-t)}\right)^{p-1}$ and find conditions under which the remaining inequality is valid. To do so, let us consider the $s$-terms. By the monotonicity of the $\Gamma$-fractions and \eqref{eq:Gautschi_s}, we have 
		\begin{multline*}
			\frac{\Gamma(y-1-s)}{\Gamma(y+s)}+ \frac{\Gamma(y+1-s)}{\Gamma(y+2+s)} \leq 2\cdot \frac{\Gamma(y-1-s)}{\Gamma(y+s)} \\
			\leq 2\cdot \frac{1}{y-1-s}\cdot \frac{1}{(y-1+s)^{2s}}\leq \frac{2}{y^{1+2s}}\cdot \frac{1}{1-\frac{1+s}{2}}\cdot \frac{1}{(1-\frac{1-s}{2})^{2s}}\\
			= \frac{2^{2+2s}}{y^{1+2s}}\cdot \frac{1}{1-s}\cdot \frac{1}{(1+s)^{2s}}.
		\end{multline*}
		Moreover, since $t(p-1)<s$, we have
		\[\sum_{y=3}^\infty y^{-1-2s+2t(p-1)}<\int_2^\infty y^{-1-2s+2t(p-1)}dy = \frac{2^{-2s+2t(p-1)}}{2s-2t(p-1)}.\]
		We conclude
		\begin{multline*}
			\left((1-2t)\frac{\Gamma(1+t)}{\Gamma(2-t)}\right)^{1-p}\cdot S_{\infty}\\
			\leq \frac{1}{(1-s)(1+s)^{2s}} \frac{2^{2+2t(p-1)}}{(1-t)^{p-1}}\cdot \frac{1}{2s-2t(p-1)},
		\end{multline*}
		and by \eqref{eq:Gautschi_t},
		\begin{align*}
			\left((1-2t)\frac{\Gamma(1+t)}{\Gamma(2-t)}\right)^{1-p}\cdot S_0\geq \frac{1}{(1+s)(1-s)^{2s}}\cdot \frac{1}{t^{p-1}}.
		\end{align*}
		Thus, $S_0\geq S_2+S_{\infty}$ is satisfied if
		
		\begin{align*}
			\left(\frac{1+s}{1-s}\right)^{1-2s}\leq \frac{1}{t^{p-1}\left(4\cdot \left(\frac{4^{t}}{1-t}\right)^{p-1}\cdot \frac{1}{2s-2t(p-1)}+ \frac{2}{\left(2-t\right)^{p-1}}\right)}
		\end{align*}
		Since the right-hand side has a singularity at $t=0$, we can continue as in Lemma~\ref{lem:superharmonicat1_p}. 
	\end{proof}
	
	\section*{Acknowledgments} We thank Matthias Keller for helpful comments.
	\printbibliography
\end{document}